\documentclass[10pt]{amsart}
\usepackage{amssymb}
\usepackage{caption}
\usepackage{subcaption}
\usepackage{amsmath}
\usepackage{amsthm}
\usepackage{amsbsy}
\usepackage{psfrag}
\usepackage{pstricks}
\usepackage{graphics}
\usepackage{graphicx}
\usepackage[utf8]{inputenc}
\usepackage{xcolor}
\usepackage{epstopdf}
\newcounter{example}[section]
\newenvironment{example}[1][]{\refstepcounter{example}\par\medskip
	\textbf{Example~\thesection.\theexample. #1} \rmfamily}{\medskip}
\theoremstyle{plain}
\newtheorem{theorem}{Theorem}[section]
\newtheorem{lemma}[theorem]{Lemma}

\newtheorem{definition}[theorem]{Definition}

\theoremstyle{remark}
\newtheorem{remark}[theorem]{Remark}

\begin{document}
\allowdisplaybreaks[4]
\numberwithin{figure}{section}
\numberwithin{table}{section}
 \numberwithin{equation}{section}
%
\title[Pointwise adaptive quadratic DG FEM for elliptic obstacle problem]
 {Supremum-norm a posteriori error control of quadratic discontinuous Galerkin methods for the obstacle problem}
 \author{Rohit Khandelwal}\thanks{The first author's work is supported  by the Council for Scientific and Industrial Research (CSIR)}
 \email{rohitkhandelwal004@gmail.com}
 \address{Department of Mathematics, Indian Institute of Technology Delhi, New Delhi - 110016}
 \author{Kamana Porwal}\thanks{The second author's work is supported by Council for Scientific and Industrial Research (CSIR) Extramural Research Grant}
 \email{kamana@maths.iitd.ac.in}
 \address{Department of Mathematics, Indian Institute of Technology Delhi, New Delhi - 110016}
 \author{Ritesh Singla}
 \email{iritesh281@gmail.com}
 \address{Department of Mathematics, Indian Institute of Technology Delhi, New Delhi - 110016}

 \date{}
 \begin{abstract}
 	 We perform a posteriori error analysis in the supremum norm for the quadratic discontinuous Galerkin method for the elliptic obstacle problem. We define two discrete sets (motivated by Gaddam, Gudi and Kamana \cite{gaddam2021two}), one set having integral constraints and other one with the nodal constraints at the quadrature points, and discuss the pointwise reliability and efficiency of the proposed a posteriori error estimator. In the analysis, we employ a linear averaging function to transfer DG finite element space to standard conforming finite element space and exploit the sharp bounds on the Green’s function of the Poisson's problem.  Moreover, the upper and the lower barrier functions corresponding to continuous solution $u$ are constructed by modifying the conforming part of the discrete solution $u_h$ appropriately.  Finally, numerical experiments are presented to complement the theoretical results.

 \end{abstract}
\date{}
\maketitle
\allowdisplaybreaks
\def\R{\mathbb{R}}
\def\cA{\mathcal{A}}
\def\cK{\mathcal{K}}
\def\cN{\mathcal{N}}
\def\p{\partial}
\def\O{\Omega}
\def\Bar {\overline}
\def\bbP{\mathbb{P}}
\def\cV{\mathcal{V}}
\def\cM{\mathcal{M}}
\def\cT{\mathcal{T}}
\def\cE{\mathcal{E}}
\def \cW{\mathcal{W}}
\def \cV{\mathcal{V}}
\def\bF{\mathbb{F}}
\def\bC{\mathbb{C}}
\def\bN{\mathbb{N}}
\def\ssT{{\scriptscriptstyle T}}
\def\HT{{H^2(\O,\cT_h)}}
\def\mean#1{\left\{\hskip -5pt\left\{#1\right\}\hskip -5pt\right\}}
\def\
#1{\left[\hskip -3.5pt\left[#1\right]\hskip -3.5pt\right]}
\def\smean#1{\{\hskip -3pt\{#1\}\hskip -3pt\}}
\def\sjump#1{[\hskip -1.5pt[#1]\hskip -1.5pt]}
\def\jumptwo{\jump{\frac{\p^2 u_h}{\p n^2}}}

\section{Introduction}\label{sec1}
\par
\noindent
The obstacle problem, often considered as a prototype for a class of free boundary problems, models many phenomena such as phase transitions, jet flow, and
gas expansion in a porous medium (see \cite{rodrigues1987obstacle}). The elliptic obstacle problem is a nonlinear model that describes the vertical movement of a object restricted to lie above a barrier (obstacle) while subjected to a vertical force (with suitable boundary conditions). In general, in the obstacle problem, the given domain can be viewed as the union of the contact, non-contact and free boundary regions, which also play a crucial role in determining the regularity of the continuous solution $u$. Moreover, the location of the free boundary (the boundary of the domain where the object touches the given obstacle) is not a priori known, and therefore, it forms a part of the numerical approximation. This makes the finite element approximation of the obstacle problem an interesting subject as it offers challenges both in the theory and computation. The finite element analysis to the obstacle problem germinated in 1970s (see \cite{Falk:1974:VI,brezzi1977error}). Subsequently, there has been a lot of work on the development and analysis of finite element methods for the obstacle problem (see \cite{glowinski1984numerical,hoppe1994adaptive,brenner2012finite}).
\par
\noindent 
Adaptive finite element methods (AFEM) are useful for achieving the better accuracy of numerical solution of the obstacle problem. The key tool in the adaptive refinement algorithms is a posteriori error estimator i.e., the (globally) reliable and an (locally) efficient error estimator.  AFEM are well built for the residual-type estimators \cite{verfurth1996review,ainsworth2011posteriori,bangerth2013adaptive} and for the goal oriented dual approach for solving partial differential equations (see \cite{pollock2012convergence} and references therein). We refer \cite{johnson1992adaptive,chen2000residual,veeser2001efficient,bartels2004averaging,braess2005posteriori,NPZ:2010:VI,GG:2018:InObst} for a posteriori error analysis of linear finite element methods for the elliptic obstacle problem. The articles \cite{wang2002quadratic} and \cite{gudi2015reliable} discussed a priori and a posteriori error analysis using conforming quadratic finite element method (FEM), respectively. The article \cite{gaddam2018bubbles} presented a priori and a posteriori error control for the three dimensional elliptic obstacle problem using quadratic conforming FEM. Therein, the analysis is carried out by enhancing the quadratic finite element space with element bubble functions. Note that, all these previous mentioned articles deal with the energy norm estimates. 

In the recent years, AFEM techniques for controlling pointwise $(L_{\infty})$ error for the elliptic problems have been the subject of several recent publications (see \cite{dari1999maximum,demlow2007local,demlow2012pointwise}). For the conforming linear finite element approximations of the elliptic obstacle problem in the maximum ($L_{\infty}$) norm, we refer to works \cite{nochetto2003pointwise,nochetto2005fully}, whereas in \cite{KP:2022:Obstacle}, pointwise a posteriori error estimates are derived using quadratic conforming FEM for the elliptic obstacle problem. 

\par
\noindent For the numerical approximation of a range of problems, discontinuous Galerkin (DG) finite element methods have proven appealing in the past decade. We refer the articles \cite{gudi2014posteriori,TG:2014:VIDG1} for a posteriori error analysis using DG FEM in the energy norm for the elliptic obstacle problem. In \cite{BGP:2021}, the authors have derived pointwise a posteriori estimates of interior penalty method using linear finite elements for the obstacle problem.
Recently, in \cite{gaddam2021two}, the authors present two new ideas to solve the obstacle problem using discontinuous Galerkin FEM in the energy norm. 
The goal of this work is to derive the reliable and efficient pointwise a posteriori error estimates for the elliptic obstacle problem using quadratic discontinuous Galerkin finite element methods. \\

\par
\noindent
This paper is organized as follows. In Section \ref{sec2}, we introduce the continuous obstacle problem and the Green's function for the Poisson's problem. Section \ref{sec3} is devoted to notations and preliminaries. The main ingredients of the Section \ref{sec3} are the  local projection (interpolation) operator (introduced by Demlow \cite{demlow2012pointwise}) and the averaging (enriching) operator (motivated by Brenner \cite{brenner1996two})  and we discuss their  approximation properties which are useful in the subsequent analysis. In Section \ref{sec4}, we define two discrete convex sets $\cK_h^{(t)}$ for $t=1,2$ and introduce the discrete problems corresponding to each $\cK_h^{(t)}$. Section \ref{sec5} is dedicated to  a posteriori error analysis in the supremum norm. The key ingredients in proving the reliable and an efficient estimates in the supremum norm are the upper and lower barriers of the continuous solution $u$ and bounds on the Galerkin functional. We present the results of the numerical experiments in  Section \ref{sec6} illustrating the reliability and efficiency of the proposed error estimator.

\section{The continuous problem} \label{sec2} The model assumptions which we considered are the following.
\large \begin{itemize}
	
	\item $\Omega$ is a bounded polygonal domain in  $\mathbb{R}^d$ with boundary $\partial \Omega,$ where $d=\{2,3\}.$ 
	\item $V=H^1_0(\Omega)$ and $f \in L_{\infty}(\Omega).$
	\item The obstacle $\chi \in H^1(\Omega) \cap C(\bar{\Omega})$ be such that $\chi \leq 0$ on $\partial \Omega.$
	\item $a(u,v)= \int_{\Omega} \nabla u\cdot \nabla v~dx$ and $(f,v)= \int_{\Omega} fv~dx$.
	\item $\mathcal{K}=\{v \in V:v \geq \chi~ ~a.e.~\mbox{in}~ \Omega\}$.
	\item $(\cdot,\cdot)$ denotes the $L_2(\Omega)$ inner product.
	\item For any $1 \leq p \leq \infty$ and $D \subset \Omega$, we denote the norm on the space  $L_p(D)$ by $||\cdot||_{L_p(\Omega)}.$\\
\end{itemize}
The continuous problem is to find $u \in \mathcal{K}$ such that the following holds
\begin{equation} \label{SSSS}
a(u,u-v) \leq (f,u-v)~~~\forall v \in \mathcal{K}.\\
\end{equation}
 $\mathbf{\underline{Observations.}}$
\begin{itemize}
	\item $a(u,v)$ is a continuous bilinear form over $V$, symmetric and $V$-elliptic \cite{s1989topics}.
	\item $\mathcal{K}$ is a non empty, closed and convex subset of $V$ \cite{s1989topics}.
\end{itemize}
By the theory of elliptic variational inequality \cite{glowinski1980numerical}, the problem (\ref{SSSS}) has a unique solution.\\
Define $\sigma\in H^{-1}(\Omega)$ by
\begin{align} \label{CCC}
\langle \sigma, v \rangle = (f, v) - a(u,v)~ \mbox{for all}~ v \in V,
\end{align}
where $\langle \cdot, \cdot \rangle$ denotes the duality pair of $H^{-1}(\Omega)$ and $H^1_0(\Omega).$ The norm $||\cdot||_{-1}$ on $H^{-1}(\Omega)$ is defined by
\begin{align}
||\phi||_{-1}:= \underset{v\in H^1_0(\Omega),v \neq 0}{sup} \frac{\langle \phi,v \rangle}{||\nabla v||_{L_2(\Omega)}}.
\end{align}

\noindent
   From the equations \eqref{SSSS} and \eqref{CCC}, it follows
\begin{equation} \label{prop}
\langle \sigma ,v-u \rangle \leq 0 \quad \forall v \in \cK.
\end{equation}
\begin{remark} \label{sign}
	Let $v=u+\phi \in \cK$ in \eqref{prop} for $0 \leq \phi \in H^1_0(\Omega)$, we get $\sigma \leq 0.$ Moreover, if $u > \chi$ on some open subset $A \subset \Omega $, then $\sigma =0$ in $A$ ( see \cite{kinderlehrer2000introduction} for details). Moreover, if $\Omega$ is convex domain \cite{frehse1978smoothness,caffareli1980potential} and if $f \in L_{\infty}(\Omega)$ and $\chi \in H^1(\Omega) \cap C^{0,\beta}(\bar{\Omega})$, $0 < \beta \leq 1$, then $u \in H^1_0(\Omega) \cap C^{0,\alpha}(\bar{\Omega})$ where $0 < \alpha \leq \beta$, i.e., $u$ is H\"older continuous.
\end{remark}
\begin{remark}
	In the view of Riesz representation theorem	[\cite{kinderlehrer2000introduction}, Theorem 6.9], we note that $\sigma$ can be treated as a radon measure which satisfies
	\begin{align*}
	\langle \sigma, v \rangle = \int_{\Omega} v d\sigma,
	\end{align*}
	where $v$ is a continuous function having compact support.
\end{remark}

\noindent

\subsection{Green's function for linear Poisson problem} 
The Green's function is commonly employed while performing a posteriori error analysis in the maximum norm. We refer \cite{nochetto1995pointwise,dari1999maximum,nochetto2003pointwise,demlow2012pointwise,demlow2016maximum} for some work in this area. We state the existence of the Green's function for the laplacian operator and we refer \cite{gruter1982green,hofmann2007green} for the proof.
\begin{theorem} \label{Green}
There exist a unique function $G : \Omega \times \Omega \rightarrow \mathbb{R} \cup \{\infty\}$ such that for any $y \in \Omega$, the following holds
\begin{align} \label{greeen}
a(G(y,\cdot), \phi) &= \phi(y) \quad \text{in}~\Omega, \\ G(y, \cdot)&= 0 \quad \text{on}~ \partial \Omega, \nonumber
\end{align}
for any $\phi \in D(\Omega)$. Moreover for any $q>d$ and $g \in W^{1,q}(\Omega) \cap H^1_0(\Omega)$, we have
\begin{align} \label{eqq}
g(y)= a(w,G(y, \cdot)).\end{align} 
\end{theorem}
\noindent
Note that $G$ has a singularity at $y \in \Omega$. The next lemma states some regularity estimates for the Green's function $G$ which takes into account it's singular behavior. We refer the articles \cite{demlow2012pointwise,demlow2016maximum} for the details.
\begin{lemma} \label{Grn}
	Let $G$ be defined  in equation \eqref{greeen}, then for any $y \in \Omega$, it holds that
	\begin{align}
	|G(y,\cdot)|_{0,1,\Omega} + |G(y,\cdot)|_{0,\frac{d}{d-1},\Omega} +|G(y,\cdot)|_{1,1,\Omega} \lesssim 1.
	\end{align}
	Moreover, for the ball $B(y,R)$ having centered at $y$ and of radius $R$, the following estimates hold
	\begin{align*}
 	|G(y,\cdot)|_{1,\frac{d}{d-1},(\Omega \setminus B(y,R))} & \lesssim \text{log}(2+R^{-1}), \\ 	|G(y,\cdot)|_{2,1,(\Omega \setminus B(y,R))} & \lesssim \text{log}(2+R^{-1}), \\	|G(y,\cdot)|_{0,1,(\Omega \cap B(y,R))} &\lesssim R^2 \text{log}(2+R^{-1})^{\alpha_d}, \\ 	|G(y,\cdot)|_{1,1,(\Omega \cap B(y,R))} & \lesssim R,
	\end{align*}
	where $\alpha_2=1$ and $\alpha_3=0.$
\end{lemma}
\section{Notations and Preliminaries} \label{sec3}
\par
\noindent
In this section, we collect some notations and prerequisites for the upcoming analysis.
\begin{itemize}
	\item $\cT_h$ be a regular triangulation of $\Omega$ (see \cite{ciarlet2002finite})
	\item $T \in \cT_h$ denotes a triangle for $d=2$ or tetrahedron for $d=3$
	\item $h_T$ be the diameter of $T$~~, $h_{min}= \text{min} \{h_T~:~T \in \cT_h\}.$
     \item $\mathcal{V}_T$ denotes the set of all vertices of $T$
    \item  $\mathcal{M}_T$ denotes the set of all midpoints on the edges/faces of $T$
    \item $\mathcal{V}_h^i$denotes the  set of all interior vertices in $\cT_h$
    \item  $\mathcal{V}_h^b$ denotes the set of all boundary vertices of $\cT_h$
    \item $\mathcal{V}_h= \mathcal{V}^i_h \cup \mathcal{V}^b_h$ is the set of all the vertices of $\cT_h$
	\item $\mathcal{E}^i_h$ is the set of all interior edges/faces of $\cT_h$
	\item $\mathcal{E}^b_h$ is the set of all the boundary edges/faces of $\cT_h$
	\item $\mathcal{E}_h= \mathcal{E}^i_h \cup \mathcal{E}^b_h$ is the set of all the edges/faces of $\cT_h$
	\item $h= \max \{h_T~:~T \in \cT_h\}$
	\item $|e|$ is a length of an edge/face $e \in \mathcal{E}_h$
	\item $\cT_p$ denotes the set of all elements in $\cT_h$ that share the common vertex $p$
	\item $\cT_e$ denotes the set of all elements in $\cT_h$ that share the common edge/face $e$
	\item $|A|$ denotes the volume of the set $A \subset \Omega$
	\item $h_e$ is the length of an edge/face $e \in \mathcal{E}_h$
 
	\item $A \lesssim B$ means that there exist a positive constant $C$ (independent of the solution and the mesh parameter) such that $A \leq CB$
	\item $supp (g)$ denotes the support of the function $g$
	\item $P_k(T)$ is the linear space of polynomials of degree less than or equal to $k$ over $T$, where $k \in \mathbb{N} \cup \{0\}$
\end{itemize}

\noindent
We need to introduce the average and jump of the discontinous functions. First, we define the broken Sobolev space
\begin{align*}
H^1(\Omega, \cT_h)=\{v \in L_2(\Omega)~:~v|_T \in H^1(T) \quad \forall T \in \cT_h\}.
\end{align*}
Let $e$ be an interior edge/face in $\cT_h$, then there exist two elements $T^+$ and $T^-$ such that $e = \partial T^+ \cap \partial T^-$. Let $n^+$ be an unit outward normal pointing from $T^+$ to $T^-$, then we have $n^-=-n^+$. Hence, the jump and average of $v \in H^1(\Omega, \cT_h)$ on an edge/face $e$ is defined by
\begin{align*}
\{\{v\}\}&= \frac{v^++v^-}{2} \\ [[v]]&= v^+n^+ + v^-n^-,
\end{align*} respectively where $v^+=v|_{T^+}$ and $v^-=v|_{T^-}$. We follow the same idea to define the jump and average for a vectorvalued function $q \in [H^1(\Omega, \cT_h)]^d$ on interior edge/face $e$
\begin{align*}
\{\{q\}\}&= \frac{q^++q^-}{2} \\ [[q]]&= q^+\cdot n^+ + q^-\cdot n^-.
\end{align*} 
For $e \in \cE^b_h$,  let $n^e$ be an outward unit normal to an element $T$
 such that $\partial T \cap \partial \Omega =e$, we define for $v \in H^1(\Omega, \cT_h) $
 \begin{align*}
 [[v]]=v n^e \quad \text{and} \quad \{\{v\}\} = v
 \end{align*} and for $q \in [H^1(\Omega, \cT_h)]^d$, we set
  \begin{align*}
 [[q]]=q \cdot n^e \quad \text{and} \quad \{\{q\}\} = q.
 \end{align*}
 Let $V_h$ denotes the quadratic DG finite element space which is defined by
 \begin{equation*}
 V_{h}=\{v_h\in L_2({\Omega}) : v_h|_T\in P_2(T)\,\,\forall\, T\in
 {\mathcal T}_h\},
 \end{equation*}
 and $V_c$ denotes the conforming quadratic finite element space, i.e., $V_c:=V_h \cap H^1_0(\Omega)$. 
 The following estimates will be crucial in the upcoming analysis.
 
 \begin{lemma} (Trace inequality \cite{brenner2007mathematical}) \label{trace}
 	Let $\phi \in H^1(T)$ for $T \in \cT_h$ and let $e \in \mathcal{E}_h$ be an edge/face of $T$, then for $1 \leq p < \infty$ the following holds
 	\begin{equation}
 	||\phi||^p_{L_p(e)} \lesssim   h^{-1}_p\Big(|\phi||^p_{L_p(T)}+ h_e^{p}||\nabla \phi||^p_{L_p(T)}\Big ). \label{3.1}
 	\end{equation}
 \end{lemma}

 \begin{lemma} (Inverse inequalities \cite{ciarlet2002finite})\label{inverse}
 	For $v_h \in V_h$ and $1 \leq p,q \leq \infty$, it holds that
 	\begin{align*}
 	||v_h||_{m,p,T} & \lesssim h_T^{d(\frac{1}{p}-\frac{1}{q})}  h^{l-m}_T ||v_h||_{l,q,T} \\ 
 	||v_h||_{L_{\infty}(e)} &\lesssim  h_e^{\frac{1-d}{2}}||v_h||_{L_2(e)}
 	\end{align*}
 	where $T \in \cT_h$, $e \in \mathcal{E}_h$ is an edge/face of $T \in \cT_h$.
 \end{lemma}
%

%
 \begin{lemma} (Poincare's type inequality [\cite{evans1998partial}, Theorem 3, page 279])\label{poincare}
 	Let $\Omega$ be a bounded open subset of $\mathbb{R}^d$. Suppose $v \in W^{1,p}_0(\Omega)$ for some $1 \leq p <n$ and $p^*= \frac{pd}{d-p}$ be the Sobolev conjugate of $p$, then there exist a positive constant $C$ (depending on $p,q,d$ and $\Omega$) such that
 	\begin{align}
 	||u||_{L_q(\Omega)} \leq C ||Du||_{L_p(\Omega)} \quad \text{for each}~ q \in [1, p^*],
 	\end{align}
 	where $Du$ denotes the first order distributional derivative of $u$.
 \end{lemma}
 
\subsection{Local Projection Operator} Here, we provide the explicit definition of the projection operator which is used in our analysis. Let $\Pi_h : L_1(\Omega) \rightarrow V_h$ be the local projection operator \cite{demlow2012pointwise} which is motivated by the classical Scott and Zhang interpolation operator \cite{scott1990finite}. For any $T\in \cT_h$, let $(\phi_p^T)_{p \in \mathcal{V}_T \cup \mathcal{M}_T}$ be the corresponding Lagrange basis. We define
\begin{align}
\Pi_hv|_T(x)=\Pi_Tv(x)= \sum_{p \in \mathcal{V}_T \cup \mathcal{M}_T} \phi_p^T(x) \int_{T} \psi_p^T(z)v(z)dz,
\end{align}

 \noindent
where $(\psi_p^T)_{p \in \mathcal{V}_T \cup \mathcal{M}_T}$ is the corresponding $L_2(T)$ dual basis of  $(\phi_p^T)_{p \in \mathcal{V}_T \cup \mathcal{M}_T}$.
\begin{remark} \label{proj}
Using the definition of dual basis \cite{demlow2012pointwise}, we have	$ \int_{T} \psi_p^T(x)\phi_r^T(x)dx=\delta_{pr}$, therefore $\Pi_hv=v~~ \forall v \in V_h$.
\end{remark} 
\noindent

Further, we can prove the following result by using Bramble Hilbert lemma. These are standard approximation and stability results hence we skip the proof ( refer the article \cite{demlow2012pointwise} for the details).
\begin{lemma} \label{approx}
Let $s \in \mathbb{N} \cup \{0\}$ such that $0 \leq s \leq 3$, then for any $T \in \cT_h$
	\begin{align}
	|\Pi_h \psi|_{W^{s,1}(T)} &\leq C |\psi|_{W^{s,1}(T^*)}, \label{DR1}\\
	||\psi -\Pi_h \psi||_{W^{r,1}(T)} &\leq C h_T^{s-r}|\psi|_{W^{s,1}(T^*)},~\text{where}~ r=0,1,2, \label{DR2}
	\end{align}
	where $\psi \in W^{1,p}(\Omega)$ and $T^*= \{T' \in \cT_h: T' \cap T \neq \varnothing \}$.
\end{lemma} 

\subsection{Averaging Operator} We relate the discrete space and continuous space through an averaging (enriching) map. It is well known that an enriching map $E_h:V_h \to V_c$ plays an crucial role in  a posteriori error analysis of DG finite element methods. We define $E_h$ by standard averaging technique \cite{brenner1999convergence}. Let $v_h \in V_h$, define $E_hv_h \in V_c$ as follows:\\
For an interior node $p \in \mathcal{V}_h^i \cup \mathcal{M}_h^i$, we define\\
\begin{align}
E_hv_h(p)=\frac{1}{|\cT_p|} \sum_{T \in \cT_p} v_h|_T(p),
\end{align}
where $|\cT_p|$ denotes the cardinality of $\cT_p,$ and for a boundary node  $p \in \mathcal{V}_h^b \cup \mathcal{M}_h^b$ we set $E_hv_h(p)=0.$ From the standard scaling arguments \cite{brenner2007mathematical} and inverse inequalities, we prove the following approximation properties.

\begin{lemma} \label{lem:average}
	Let $v \in V_h$, then 
	\begin{align}
	||E_hv-v||_{L_\infty(\Omega_h)} & \lesssim ||\sjump{v}||_{L_\infty(\mathcal{E}_h)}, \label{pp}\\
	\underset{T\in \mathcal{T}_h}{max}||h_T\nabla(E_hv-v)||_{L_\infty(T)} & \lesssim  ||\sjump{v}||_{L_\infty(\mathcal{E}_h)}. \label{ppp}
	\end{align}
	\begin{proof}
		Let $T \in \cT_h$ be arbitrary and using scaling arguments,
		\begin{align}
		||E_hv-v||_{L_{\infty}(T)} &=  || \sum_{p \in \mathcal{V}_T}(E_hv-v)(p)\phi^T_p +  \sum_{p \in \mathcal{\mathcal{M}}_T}(E_hv-v)(p)\phi^T_p||, \\ & \lesssim \sum_{p \in \mathcal{V}_T \cup \mathcal{M}_T} |(E_hv-v)_T(p)|,
		\end{align}
		where, we have used $||\phi_p||_{L_{\infty}(T)} \lesssim 1 ~~ \forall T \in \cT_h, \forall p \in \mathcal{V}_T \cup \mathcal{M}_T. $ 

		Let $p \in \mathcal{V}_T \cap \mathcal{V}_h^i$ or $p \in \mathcal{M}_T \cap \mathcal{M}_h^i$, then using the definition of $E_h$ we have
		\begin{align*}
		|(E_hv-v)_T(p)| &= \frac{1}{|\cT_p|} \left|\sum_{T^* \in \cT_p} v|_{T^*}(p)-v|_T(p)\right|,\\
		&=\frac{1}{|\cT_p|} \left|\sum_{T^* \in \cT_p}\Big(v|_{T^*}(p)-v|_T(p)\Big)\right|.
		\end{align*}
		Note that $T^*$ can be connected to T by a chain of elements in $\cT_p,$ where we can find a common edge/face between any two consecutive elements i.e., there exist $T_1, T_2,\cdots,T_n \in \cT_p~~ (n \leq |\cT_p|)$ such that $T_1=T^*,T_n=T$, $T_i~ \mbox{and}~ T_{i+1}$ share a common edge/face $e_i$ $\forall 1 \leq i \leq n-1$. Therefore, we can write
		\begin{align*}
		v|_{T^*}(p)-v|_T(p)&=v|_{T_1}(p)-v|_{T_2}(p)+v|_{T_2}(p)-v|_{T_3}(p)\\&\qquad+ v|_{T_3}(p)-\cdots+v|_{T_{n-2}}(p)-v|_{T_{n-1}}(p)+v|_{T_{n-1}}(p)-v|_{T_n}(p),
		\end{align*}
		then we have
		\begin{align*}
		|v|_{T^*}(p)-v|_T(p)| \lesssim \sum_{e \in \cE_p} ||\sjump{v}||_{L_{\infty}(e)},
		\end{align*}
		which leads to \eqref{pp} and we can prove \eqref{pp} for the boundary node as well on the similar lines. The estimate \eqref{ppp} can be realized with a use of Lemma \ref{inverse}.
	\end{proof}
\end{lemma}
\section{Discrete Problem} \label{sec4} 
\noindent
In this paper, we consider two different ways to define the discrete version of the set $\cK$. First, we use integral constraints which is motivated by the article \cite{gudi2015reliable} and secondly, we used the nodal constraints at some quadrature points in the discrete convex set which will be defined later in this section. Let
{
\begin{equation*}
V_0=\{v\in L_2({\Omega}) : v|_T\in P_0(T)~ \forall\, T\in
{\mathcal T}_h\},
\end{equation*}}

and for any {$v \in L_2(T)$, $T\in {\mathcal T}_h$}, define
\begin{equation} \label{formula}
    Q_T(v)= \frac{1}{|T|}\int_T v(x)~dx.
\end{equation}

Then, {$Q_h: L_2(\Omega) \rightarrow V_0$ is defined as~ $Q_h(v)|_T=Q_T(v)$ for all $v \in L_2(\Omega).$} The approximation properties for the operator $Q_h$ are stated in the next lemma.
\begin{lemma} \label{approximation}
	Let $T \in \cT_h$ and $v \in W^{s,p}(T)$ where $0 \leq s \leq 1$, then it holds that
	\begin{align}
	||v-Q_h(v)||_{L_p(T)} \lesssim  h^s_T |v|_{W^{s,p}(T)}. \label{4.2}
	\end{align}
\end{lemma} 
\noindent
The first discrete version of $\cK$ is denoted by $\cK_h^{(1)}$ which is defined by
\begin{align} \label{dset1}
\cK_h^{(1)}= \{v_h \in V_h: Q_h(v_h) \geq Q_h(\chi)\},
\end{align}

\noindent
\textbf{Discrete Problem 1.}
The discontinuous Galerkin approximation $u_h^{(1)} \in \cK_h^{(1)}$ is the solution of
\begin{equation} \label{DDD}
    \mathcal{A}_h(u_h^{(1)},v_h-u_h^{(1)}) \geq (f,v_h-u_h^{(1)})~~ \forall v_h\in \cK_h^{(1)},
\end{equation}
where $\mathcal{A}_h(\cdot,\cdot)$ is the DG bilinear form and further it can be written as
\begin{equation} \label{DP12}
    \mathcal{A}_h(v_h,w_h)=a_h(v_h,w_h) + b_h(v_h,w_h) ~\forall v_h,w_h \in V_h,
\end{equation}
with 
\begin{equation*}
    a_h(v_h,w_h)= (\nabla_h v_h, \nabla_h w_h) = \sum_{T \in \cT_h}\int_T \nabla v_h \cdot \nabla w_h~ dx,
    \end{equation*}
and the bilinear form $b_h(\cdot,\cdot)$ is defined by
\begin{align}
b_h(w,v) = - \sum_{e \in \cE_h} \int_e \{\{\nabla w\}\} [[v]]~ ds  & - \theta \sum_{e \in \cE_h} \int_e \{\{\nabla v\}\} [[w]]~ ds  \\ & + \sum_{e \in \cE_h} \int_e \frac{\eta}{h_e} [[w]] [[v]]~ ds, \quad \forall w , v \in V_h.
\end{align}
\begin{remark} \label{DGform}
	Different value of $\theta$ will give rise to different DG methods \cite{wang2011discontinuous}. In our paper, we deal with SIPG, IIPG and NIPG methods which corresponds to $\theta \in \{1,0,-1\}$, though the analysis holds for other DG methods as well  listed in \cite{gudi2014posteriori}.
 \end{remark}
 Next result is concerning the operator $\pi_h $ and it's properties.
\begin{lemma} \label{4.3}
Let $\pi_h : V_h \rightarrow V_0$ be a map defined by $\pi_h(v_h)=Q_h(v_h)$. Then $\pi_h$ is onto and hence an inverse map $\pi_h^{-1}: V_0 \rightarrow V_h$ can be defined into a subset of $V_h$ as $\pi_h^{-1}(w_h)=v_h$ where $v_h\in V_h$ with $Q_h(v_h)=w_h$ for $w_h\in V_0.$
\end{lemma}
\noindent
We refer the article \cite{gaddam2018bubbles} for the proof of the lemma. Next, with the help of Lemma \ref{4.3}, we define the first discrete Lagrange multiplier $\sigma_h^{(1)} \in V_0$ (corresponding to $\cK_h^{(1)}$) as
\begin{align} \label{DLM}
(\sigma_h^{(1)}, w_h)= (f, \pi_h^{-1} w_h)-\mathcal{A}_h(u_h,\pi_h^{-1} w_h) \quad \forall w_h \in V_0.
\end{align} In the next lemma, we derive key properties for $\sigma_h^{(1)}$.
\begin{theorem} \label{signn}
	$\sigma_h^{(1)}$ is well-defined and $\sigma_h^{(1)} \leq 0$ on $\Bar{\Omega}.$
	\begin{proof}
		Let $y_h \in V_h$ be such that $ Q_h(y_h) \geq 0$ then by putting $v_h= u_h^{(1)} + y_h$ in \eqref{DDD} we observe
		\begin{equation} \label{DP11}
		\mathcal{A}_h(u_h^{(1)},y_h) \geq (f,y_h)~for ~y_h \in V_h~ with~ Q_h(y_h) \geq 0.
		\end{equation} 
		Further, let $v_h \in V_h$ with $Q_h(v_h)=0$, then substitute $y_h=v_h$ and $y_h=-v_h$ in \eqref{DP11} to obtain
		\begin{equation} \label{DP111}
		\mathcal{A}_h(u_h^{(1)},v_h) = (f,v_h).
		\end{equation}

	Now suppose	there exist two distinct elements $z_1$ and $z_2$ $\in V_h$ such that $\pi_h(z_1)=\pi_h(z_2)=w_h$. Then $Q_h(z_1-z_2)=0$, we have $\mathcal{A}_h(u_h^{(1)},z_1-z_2) = (f,z_1-z_2)$ by \eqref{DP111} which implies $\sigma_h^{(1)}$ is well defined. Moreover, let $w_h \geq 0$ in \eqref{DLM} and using \eqref{DP11} we have that $\sigma_h^{(1)} \leq 0$ on $\bar{\Omega}$.
	\end{proof}
\end{theorem} \noindent
We will define  the following discrete contact set and discrete non-contact set for $u_h^{(1)}.$
\begin{align*}
\mathcal{C}_h^{(1)}= \{T \in \cT_h: Q_h(u_h^{(1)}) = Q_h(\chi)\},\\
\mathcal{N}_h^{(1)}= \{T \in \cT_h: Q_h(u_h^{(1)}) > Q_h(\chi)\}.
\end{align*}
We also observe that $\sigma_h^{(1)} = 0$ over $\mathcal{N}_h^{(1)}$ and it is clear (from the definition of $\sigma_h^{(1)}$) that for any $v_h \in V_h$
\begin{align} \label{sigma1}
(\sigma_h^{(1)},v_h)&=(\sigma_h^{(1)},Q_h(v_h)) \nonumber,\\
&=(f,v_h)-\mathcal{A}_h(u_h^{(1)},v_h).
\end{align}
In the next subsection, we collect some tools to define the second discrete version of $\cK$ and the corresponding discrete Lagrange multiplier $\sigma_h^{(2)}$. We start with some facts about numerical integration in $\mathbb{R}^2$ and $\mathbb{R}^3$.
\subsection{Numerical Integration} \label{integration}
\begin{itemize}
	\item $T$ be a $d$-simplex formed by the vertices $\{x_1,x_2,\cdots,x_{d+1}\}.$
	\item $G=\frac{1}{d+1} \sum_{i=1}^{d+1} x_i$ (The Centroid of $T$).
	\item  Choose $r_i= x_i~ \forall ~i \in \{1,2,\cdots,d+1\}$ for $s=1.$
	\item  Choose $r_i=bx_i + (1-b)G$ with $b=\frac{1}{\sqrt{d+2}} \forall ~i \in \{1,2,3,\cdots,d+1\}$ for $s=2.$
\end{itemize}
 
Then, the formula 
\begin{align} \label{exact}
\int_{T} p~ dx = \frac{|T|}{d+1} \sum_{i=1}^{d+1} p(r_i) \quad 
\forall p \in \mathcal{P}_s(T)
\end{align}
is exact for $d =\{2,3\}$ and $ s=\{1,2\}$.\\
\begin{definition}
	\begin{enumerate}
		\item  $\textbf{Quadrature points.}$ The Quadrature points exact for $\mathcal{P}_s(T)$ are given by $\{r_1,r_2,...,r_{d+1}\}$ (as defined earlier) where $T$ is a $d$-simplex in $\cT_h$ having vertices as $\{x_1,x_2,...,x_{d+1}\}.$
		\item $\textbf{Quadrature d-simplex.}$ Let $T \in \cT_h$ be a $d$-simplex, the Quadrature points exact for  $\mathcal{P}_s(T)$ also form a $d$-simplex which is named as Quadrature $d$-simplex and is denoted by $\mathcal{Q}_T$.
		\item $\cV_{\mathcal{Q}_T}$ denotes the set of all vertices of Quadrature $d$-simplex $\mathcal{Q}_T$ for any $T \in \cT_h.$
		\item $\cM_{\mathcal{Q}_T}$ denotes the set of all mid-points on edges/faces of Quadrature $d$-simplex $\mathcal{Q}_T$ for any $T \in \cT_h.$
		\item  $\cV_{h,\mathcal{Q}}$= $\cup_{T \in \cT_h}$ $\cV_{\mathcal{Q}_T}.$
		\item  $\cM_{h,\mathcal{Q}}$= $\cup_{T \in \cT_h}$ $\cM_{\mathcal{Q}_T}.$
	\end{enumerate}
\end{definition}
\noindent
Next, the second discrete version of $\cK$, denoted by $\cK_h^{(2)}$  is defined by
\begin{align*}
\cK_h^{(2)}=\{ v_h \in V_h: v_h(r) \geq \chi(r)~~\forall r \in   \cV_{\mathcal{Q}_T}, ~~ \forall T~ \in \cT_h \}.
\end{align*}

\textbf{Discrete Problem 2.} The discontinuous Galerkin approximation $u_h^{(2)} \in \cK_h^{(2)}$ is the solution of
\begin{equation} \label{DDDD}
\mathcal{A}_h(u_h^{(2)},v_h-u_h^{(2)}) \geq (f,v_h-u_h^{(2)})~~ \forall v_h\in \cK_h^{(2)},
\end{equation}

Next, we state the definitions of discrete contact set $\mathcal{C}_h^{(2)}$, discrete non-contact set $\mathcal{N}_h^{(2)}$ and free boundary set $\mathcal{F}_h^{(2)}$ corresponding to $u_h^{(2)}:$
\begin{align*}
\mathcal{C}_h^{(2)}= \{T \in \cT_h: u_h^{(2)}(z)=\chi(z)~\forall z  \in \cV_{\mathcal{Q}_T} \},
\end{align*}
\begin{align*}
\mathcal{N}_h^{(2)}= \{T \in \cT_h: u_h^{(2)}(z) > \chi(z)~\forall z  \in \cV_{\mathcal{Q}_T} \},
\end{align*}
and
\begin{align*}
\mathcal{F}_h^{(2)}= \{T \in \cT_h: ~\mbox{there exist} ~~z_1,z_2 \in \cV_{\mathcal{Q}_T} ~\mbox{such that}~ u_h^{(2)}(z_1) = \chi(z_1) ~\mbox{and} ~u_h^{(2)}(z_2) > \chi(z_2)\}.
\end{align*}
For the convenience, let $\{\psi_z : z \in  \cV_{h,\mathcal{Q}} \cup \cM_{h,\mathcal{Q}} \}$ be the canonical Lagrange basis of $V_h$, i.e., \\
\[
\psi_{z}(r) := \begin{cases}
1 & \text{if}~~ z = r\\
0 & \text otherwise.
\end{cases}
\]
We will need some more discrete spaces for the subsequent discussion. The space $V_h$ is decomposed as $V_h = W_h \bigoplus W_h^{c}$, where 
\begin{align*}
W_h=span \{\psi_z \in V_h: z \in \cV_{h,\mathcal{Q}} \},
\end{align*}
and the subspace $W_h^{c}$ of $V_h$ is defined as
\begin{align*}
W_h^{c}=span \{\psi_z \in V_h: z \in \cM_{h,\mathcal{Q}}  \}
\end{align*}
is the orthogonal complement of $W_h$ in $V_h$ with respect to the inner product
\begin{align*}
\langle w_h, v_h \rangle_{V_h} = \sum_{T \in \cT_h} \frac{|T|}{d+1} \left( \sum_{z \in \cV_{\mathcal{Q}_T}} w_h(z)v_h(z) + \sum_{z \in \cM_{\mathcal{Q}_T}} w_h(z)v_h(z) \right).
\end{align*}
Let
\begin{align*}
V_1= \{ v_h \in L_2(\Omega): v_h|_T \in \mathcal{P}_1(T)~ \forall ~T \in \cT_h\},
\end{align*}
and let $\{\phi_z : z \in \cV_{h,\mathcal{Q}} \}$ be the Lagrange canonical basis of $V_1$, i.e., for $q \in \cV_{h,\mathcal{Q}}$\\
\[
\phi_{z}(r) := \begin{cases}
1 & \text{if}~~ z = r,\\
0 & \text otherwise.
\end{cases}
\]
 The discrete Lagrange multiplier \textbf{$\sigma_h^{(2)} \in V_1$} is then defined by
\begin{align} \label{DLLM}
\langle \sigma_h^{(2)}, v_h \rangle_h := (f, \tilde{\Pi}_h v_h) - \mathcal{A}_h(u_h^{(2)}, \tilde{\Pi}_h v_h)~ \forall~ v_h \in V_1,
\end{align}
where 
\begin{align} \label{inner}
\langle w,v \rangle_h = \sum_{T \in \cT_h} \frac{|T|}{d+1} \sum_{z \in \cV_{\mathcal{Q}_T}} w(z)v(z)
\end{align}
and $\tilde{\Pi}_h : V_1 \rightarrow W_h$ is defined by 
\begin{align}
\tilde{\Pi}_h v = \sum_{z \in \cV_{h,\mathcal{Q}}} v(z)\psi_z,  \quad v \in V_1.
\end{align}
We list out some key observations related to $\tilde{\Pi}_h$ which is useful in later analysis. They can be verified from the definition of $\tilde{\Pi}_h$ and the Bramble Hilbert lemma \cite{ciarlet2002finite}.
\begin{itemize}
	\item  $\tilde{\Pi}_hv(z)=v(z)$ where $z \in  \cV_{h,\mathcal{Q}}$ and $v \in V_1.$
	\item $\tilde{\Pi}_h : V_1 \rightarrow W_h$ is one-one and onto and hence its inverse $\tilde{\Pi}_h^{-1} : W_h \rightarrow V_1$ exists and it is defined by\\
	\begin{align}
	\tilde{\Pi}_h^{-1}(v) = \sum_{z \in \cV_{h,\mathcal{Q}}} v(z)\phi_z, ~~~v \in W_h,~~ \phi_z \in V_1.
	\end{align}
	\item  $\tilde{\Pi}_h^{-1}$ extends to whole discrete space $V_h$ by defining
	\begin{align} \label{piinv}
	\tilde{\Pi}_h^{-1}(v) = \sum_{z \in \cV_{h,\mathcal{Q}}} v(z)\phi_z, ~~~v \in V_h.
	\end{align}
	\item  For any $v \in V_h$ we have $\tilde{\Pi}_h^{-1}v= \tilde{\Pi}_h^{-1}v_1$ where $v_1$ is the $W_h$- component of $v$.
	\item For any $v_h \in V_h$ and $T \in \cT_h$ the following approximation properties hold \cite{gaddam2021two}:
	\begin{align} 
	||v_h-\tilde{\Pi}_h^{-1}v_h||_{L_p(T)} &\lesssim  h_T||\nabla v_h||_{L_p(T)},~~~1 \leq p \leq \infty, \label{key} \\ 	  h^{-2}_T||v_h-\tilde{\Pi}_h^{-1}v_h||_{L_1(T)} &\lesssim  |v_h|_{2,1,T}~. \label{4.16}
	\end{align}
\end{itemize}
In the next theorem, we state key properties for $\sigma_h^{(2)}$ that will follow on the similar lines of  [Lemma 7,\cite{gaddam2021two}].
\begin{theorem} \label{4.6}
	\begin{align} \label{DP11111}
	\sigma_h^{(2)}(z) & \leq 0~\forall~z \in \cV_{h,\mathcal{Q}},
	\end{align}
	\begin{align} \label{DR11}
	\sigma_h^{(2)}(z) &=0~\mbox{if}~u_h^{(2)}(z)>\chi(z),~~z \in \cV_{h,\mathcal{Q}}.
	\end{align}

\end{theorem} \noindent The next lemma is an immediate consequence of the Theorem \ref{4.6}.
\begin{lemma} \label{signnn}
	There holds
	\begin{align*}
	Q_h\sigma_h^{(2)} & \leq 0 ~\mbox{everywhere on}~ \Omega,\\
	Q_h\sigma_h^{(2)} & = 0~~on ~\mathcal{N}_h^{(2)}.
	\end{align*}
where $Q_h$ is defined in equation \eqref{formula}.
\end{lemma}
\begin{proof}
	The result of this Lemma can be realized by using equations \eqref{exact}, \eqref{formula} and \eqref{DP11111}.
\end{proof}
\section{A Posteriori Error Estimates}\label{sec5}
\par
\smallskip
\noindent Define the following estimators (here $t=1,2$)
\begin{align}
\eta_1&=\underset{T\in \mathcal{T}_h}{max}~ h_T^2 \|\Delta u_h^{(t)}+f-\sigma_h^{(t)}\|_{L_{\infty}(T)}, \label{est1}\\
 \eta_2&=\underset{e \in \mathcal{E}_h^i}{max}~h_e\|\sjump{\nabla u_h^{(t)}}\|_{L_{\infty}(e)},\label{est2}\\
\eta_3 &= \|\sjump{ u_h^{(t)}}\|_{L_{\infty}(\cE_h)}, \label{est3}\\
\eta_4 &=  \underset{T\in \mathcal{T}_h}{max} ~h_T^2||\sigma_h^{(2)}-Q_h\sigma_h^{(2)}||_{L_{\infty}(T)}\label{est4},\\
 \text{ and }\quad  \eta_5 &=Osc(f,{\cT}_h),
\end{align}
where the data oscillations $Osc(f,T)$ of $f$ over
$T \subseteq\cT_h$  is defined by
$$Osc(f,T)= \min_{c\in
P_0(T)}h_T^2\|f-c\|_{L_{\infty}(T)}.$$
and we define the operator $B_h \sigma_h^{(t)} $ as\\ \[
B_h \sigma_h^{(t)} := \begin{cases}
\sigma_h^{(1)} & \text{for}~~ t=1,\\
Q_h \sigma_h^{(2)} &\text{for}~~ t=2.
\end{cases}
\]
\begin{remark} \label{sig}
	Using Lemma \ref{signnn} and Theorem \ref{signn}, we have $B_h \sigma_h^{(t)} \leq 0$ for $t=1,2$. This sign property is crucial in proving the reliability estimates in this article.
\end{remark}

\noindent
In the upcoming analysis, we would use the notation $v_c$ for $\Pi_h(v)~~~\forall v \in L_1(\Omega)$. For  a posteriori error analysis in the supremum norm, we need to introduce an extended bilinear form $\mathcal{A}^*_h(\cdot,\cdot)$ to test non-discrete functions less regular than $H^1(\Omega)$. For $w \in W^{1,p}_0(\Omega) + V_h$ and $v \in W^{1,q}_0(\Omega) + V_h$, let
\begin{align*}
\mathcal{A}^*_h(\cdot,\cdot): \left[W^{1,p}_0(\Omega) + V_h\right] \times\left[W^{1,q}_0(\Omega) + V_h\right] \longrightarrow \mathbb{R}
\end{align*} in the following way
\begin{align}
\mathcal{A}^*_h(w,v) = a_h(w,v)+ b_h^*(w,v),
\end{align}
where $ 1 \leq q < \frac{d}{d-1}$, $d <p$ and
\begin{align} \label{right}
b_h^*(w,v) = - \sum_{e \in \cE_h} \int_e \{\{\nabla w_c\}\} [[v]] ds - \theta \sum_{e \in \cE_h} \int_e \{\{\nabla v_c\}\} [[w]] ds + \sum_{e \in \cE_h} \int_e \frac{\eta}{h_e} [[w]] [[v]] ds.
\end{align}
We introduce the extended continuous Lagrange multiplier $\tilde{\sigma} \in (W^{1,q}_0(\Omega) + V_h)^*$ as
\begin{align} \label{CCCC}
\langle \tilde{\sigma} (u), v \rangle = (f, v) - a(u,v)~ \forall~ v \in W^{1,q}_0(\Omega) + V_h.
\end{align} 
\begin{remark} \label{extt}
 	Note that, for any $v \in V=H^1_0(\Omega) \subset W^{1,q}_0(\Omega)$ we have 
	$
	\langle \tilde{\sigma} (u), v \rangle = \langle \sigma(u), v \rangle.
	$
\end{remark}

 \noindent
We introduce the residual or Galerkin functional $G_h^{(t)}: W^{1,q}_0(\Omega) + V_h \rightarrow \mathbb{R}$ as 
\begin{align}\label{QWA}
\langle G_h^{(t)}, v \rangle &=a_h(u-\tilde{u}_h^{(t)},v)+\langle  \tilde{\sigma} -B_h \sigma_h^{(t)}, v\rangle \quad \forall
v\in W^{1,q}_0(\Omega) + V_h \\
&=a_h(u,v)- a_h(\tilde{u}_h^{(t)},v) +\langle  \tilde{\sigma} , v\rangle-\langle B_h \sigma_h^{(t)}, v\rangle \nonumber \quad\\
&=(f,v)- a_h(\tilde{u}_h^{(t)},v) -(B_h \sigma_h^{(t)}, v) \nonumber \quad \text{(using remark (\ref{extt}))},
\end{align}
\noindent
where $(B_h \sigma_h^{(t)}, v) = \langle B_h \sigma_h^{(t)}, v \rangle~~ \forall v \in W^{1,q}_0(\Omega) + V_h $ and $\tilde{u}_h^{(t)}$ stands for $E_h u_h^{(t)}$. Next, we introduce the corrector function $w^{(t)}$ which is defined by 
\begin{align}\label{FFF}
\int_{\Omega} \nabla w^{(t)} \cdot \nabla v~dx = \langle G_h^{(t)}, v \rangle \quad \forall v\in H^1_0(\Omega)
\end{align}
and the wellposedness of the problem \eqref{FFF} follows from Lax-Milgram lemma \cite{brenner2007mathematical}. This function $w^{(t)}$  plays a crucial role in defining the upper and lower barriers of the  continuous solution $u$. We modify the conforming part of discrete solution $\tilde{u}_h^{(t)}$ by adding this correction function $w^{(t)}$ appropriately. As mentioned previously, this approach is slightly different from the article \cite{nochetto2003pointwise,BGP:2021}. The continuous maximum principle is employed in order to bound the error term $||u-{u}_h^{(t)}||_{L_{\infty}(\Omega)}$. 
\subsection{Upper and Lower Barriers of the solution $u$} We define the  barrier functions for the solution $u$ as,
\begin{align}
u_{^{(t)}}^{*}&=\tilde{u}_h^{(t)}+w^{(t)}+\|w^{(t)}\|_{L_{\infty}(\Omega)}+\|(\chi-\tilde{u}_h^{(t)})^{+}\|_{L_{\infty}(\Omega)}, \label{AQAZ}\\
u^{^{(t)}}_{*}&=\tilde{u}_h^{(t)}+w^{(t)}-\|w^{(t)}\|_{L_{\infty}(\Omega)}-\|(\tilde{u}_h^{(t)}-\chi)^+\|_{L_{\infty}\{B_h \sigma_h^{(t)}<0\}}. \label{AQA}
\end{align}
 In the next lemma, we prove that $u_{^{(t)}}^{*}$ and $u^{^{(t)}}_{*}$ are the upper and lower bounds of the continuous solution $u$.
\begin{lemma} \label{lem:barrier}
Let $u_{^{(t)}}^{*}$ and $u^{^{(t)}}_{*}$ be as defined in \eqref{AQAZ} and \eqref{AQA}, respectively. Then 
\begin{align*}
u \leq u_{^{(t)}}^{*}  \qquad \text{and}\qquad u^{^{(t)}}_{*}\leq u.
\end{align*}
\end{lemma}
\begin{proof}
(1) First, we prove $u^{^{(t)}}_{*}\leq u$. Let $v:=$~max$\{u^{{(t)}}_{*}-u,0\}$. Note that $u^{^{(t)}}_{*} \leq u$ is equivalent to $v=0$ in $\Omega$, therefore we will show $v=0$ in $\Omega.$ Since $u,w^{{(t)}},\tilde{u}_h^{(t)} \in H_0^1(\Omega)$ and we have
\begin{align*}
(u^{(t)}_{*}-u)|_{\partial \Omega} &=\tilde{u}_h^{(t)}+w^{(t)}-\|w^{(t)}\|_{L_{\infty}(\Omega)}-\|(\tilde{u}_h^{(t)}-\chi)^+\|_{L_{\infty}\{B_h \sigma_h^{(t)}<0\}}-u, \\
& \leq 0.
\end{align*}
Therefore, $v=0$ on $\partial \Omega$. By Poincare inequality \cite{s1989topics}, it is sufficient to show that $||\nabla v||_{L_2(\Omega)}=0$. A use of equations \eqref{AQAZ}, \eqref{FFF}, \eqref{QWA} and remark (\ref{sign}) yields
\begin{align*}
\int_{\Omega} |\nabla v|^2~dx &= \int_{\Omega}\nabla(u^{(t)}_{*}-u)\cdot \nabla v ~dx, \\ &= \int_{\Omega}\nabla(\tilde{u}_h^{(t)}-u)\cdot \nabla v ~dx +\int_{\Omega} \nabla w^{(t)} \cdot \nabla v~dx, \\
&= \int_{\Omega} \nabla (\tilde{u}_h^{(t)}-u) \cdot \nabla v~dx +\langle G_h^{(t)}, v \rangle,\\
&= (\sigma- B_h \sigma_h^{(t)},v) \leq -\int_{\Omega} B_h {\sigma}_h^{(t)}~ v ~dx.
\end{align*}
It suffices to prove $\int_{\Omega} B_h \sigma_h^{(t)}~ v~dx=0$. Firstly, we show that 
 $B_h \sigma_h^{(t)}=0$ on $T \in \cT_h$ if there exist $x \in T$ such that $v(x)>0$. Suppose, by contradiction there exist $x \in int(T)$ such that $v(x)>0$ and $B_h \sigma^{(t)}_h(x)<0$. We have
 \begin{align*}
 v(x)>0 \implies u(x) &< u^{(t)}_{*}(x) \leq \tilde{u}_h^{(t)}(x)-\|(\tilde{u}_h^{(t)}-\chi)^+\|_{L_{\infty}\{B_h \sigma_h^{(t)}<0\}}, \\ & \leq \tilde{u}_h^{(t)}(x)-(\tilde{u}_h^{(t)}-\chi)^+(x), \\ 
 & \leq \chi(x),
 \end{align*}
 which is a contradiction as $u \in \cK$. Then we obtain
 \begin{align*}
 \int_{\Omega}B_h {\sigma}_h^{(t)} v~dx &= \sum_{T \in \cT_h} \int_{T} B_h {\sigma}_h^{(t)} v~dx \\
 &=0.
 \end{align*}
 Hence we get the desired result. Next, we prove that $u \leq u_{^{(t)}}^{*}$.\\

 (2) ~Let $v:=$ max$\{u-u_{(t)}^*,0\}$. We claim that $v=0$ in $\Omega$. First, observe that $v|_{\partial \Omega}=0$ since
\begin{align*}
(u-u_{(t)}^{*})|_{\partial \Omega}&=u-\tilde{u}_h^{(t)}-w^{(t)}-\|w^{(t)}\|_{L_{\infty}(\Omega)}-\|(\chi-\tilde{u}_h^{(t)})^{+}\|_{L_{\infty}(\Omega)},\\
& \leq 0.
\end{align*}
Our claim will be true if we can show that $||\nabla v||_{L_2(\Omega)}=0$. Employing equations \eqref{FFF}, \eqref{QWA} and remark (\ref{sig}), we find
\begin{align*}
\int_{\Omega}|\nabla v|^2 ~dx&= \int_{\Omega} \nabla u \cdot \nabla v~dx - \int_{\Omega} \nabla u_{(t)}^{*} \cdot \nabla v~dx, \\
&=\int_{\Omega} \nabla u \cdot \nabla v~dx -\int_{\Omega} \nabla (\tilde{u}_h^{(t)}+w^{(t)}) \cdot \nabla v~dx, \\
&=\int_{\Omega} \nabla u \cdot \nabla v~dx -\int_{\Omega} \nabla \tilde{u}_h^{(t)} \cdot \nabla v~dx -\langle G_h^{(t)}, v \rangle,\\
&= (B_h \sigma_h^{(t)}-\sigma, v) \leq (-\sigma, v) = -\int_{\Omega} v ~d{\sigma}.
\end{align*}
We further show that $supp(v)$ and  $supp(\sigma)$ are two disjoint sets which would imply that $\int_{\Omega} v~ d{\sigma}=0$ and hence, the claim holds. Let $x \in \Omega$ be such that $v(x)>0$,
\begin{align*}
\implies u(x)&> u^{*}_{(t)}(x)=\tilde{u}_h^{(t)}(x)+w^{(t)}(x)+\|w^{(t)}\|_{L_{\infty}(\Omega)}+\|(\chi-\tilde{u}_h^{(t)})^{+}\|_{L_{\infty}(\Omega)}.\\
&\geq \tilde{u}_h^{(t)}(x)+\|(\chi-\tilde{u}_h^{(t)})^{+}\|_{L_{\infty}(\Omega)},\\ & \geq \tilde{u}_h^{(t)}(x)+(\chi-\tilde{u}_h^{(t)})^{+}(x), \\ 
& \geq \chi(x),
\end{align*}
$\implies$ $supp(v) \subset \{u > \chi\}$ and $supp(\sigma) \subset \{u=\chi \}$ (remark (\ref{sign})). Therefore, the proof of the lemma follows.  \end{proof} 
\noindent 
Next, we use Lemma \ref{lem:barrier} to obtain an estimate for the pointwise error $||u-u_h^{(t)}||_{L_{\infty}(\Omega)}$.
\begin{lemma} \label{reliability}
	It holds that
	\begin{align} \label{theorem}
	||u-u_h^{(t)}||_{L_{\infty}(\Omega)} \lesssim ||w^{(t)}||_{L_{\infty}(\Omega)} + \|(\chi-u_h^{(t)})^{+}\|_{L_{\infty}(\Omega)}+
		\|(u_h^{(t)}-\chi)^+\|_{L_{\infty}\{B_h \sigma_h^{(t)}<0\}} + \eta_3.
	\end{align}
\end{lemma}
\begin{proof}
		In the view of Lemma \ref{lem:average}, we have
	\begin{align}
	\|(\chi-\tilde{u}_h^{(t)})^{+}\|_{L_{\infty}(\Omega)} & \lesssim 	\|(\chi-u_h^{(t)})^{+}\|_{L_{\infty}(\Omega)}  + 	\|u_h^{(t)}-\tilde{u}_h^{(t)}\|_{L_{\infty}(\cT_h)},   \nonumber\\ & \lesssim \|(\chi-u_h^{(t)})^{+}\|_{L_{\infty}(\Omega)}  + \eta_3 \label{5.9}
	\end{align}
	and 
	\begin{align}	
	\|(\tilde{u}_h^{(t)}-\chi)^+\|_{L_{\infty}\{B_h \sigma_h^{(t)}<0\}} & \lesssim 	\|(u_h^{(t)}-\chi)^+\|_{L_{\infty}\{B_h \sigma_h^{(t)}<0\}} + 	\|\tilde{u}_h^{(t)}-u_h^{(t)}\|_{L_{\infty}\{B_h \sigma_h^{(t)}<0\}}, \nonumber \\ & \lesssim 	\|(u_h^{(t)}-\chi)^+\|_{L_{\infty}\{B_h \sigma_h^{(t)}<0\}} + \eta_3. \label{5.10}
	\end{align}
	A use of triangle inequality yields
	\begin{align} \label{5.11}
		||u-u_h^{(t)}||_{L_{\infty}(\Omega)}  \lesssim 	||u-\tilde{u}_h^{(t)}||_{L_{\infty}(\Omega)}  + 	||\tilde{u}_h^{(t)}-u_h^{(t)}||_{L_{\infty}(\Omega)}.
	\end{align}
	From Lemma \ref{lem:barrier}, we have

	\begin{align*}
	||u-\tilde{u}_h^{(t)}||_{L_{\infty}(\Omega)} \lesssim 2||w^{(t)}||_{L_{\infty}(\Omega)} + \|(\chi-\tilde{u}_h^{(t)})^{+}\|_{L_{\infty}(\Omega)}+
	\|(\tilde{u}_h^{(t)}-\chi)^+\|_{L_{\infty}\{B_h \sigma_h^{(t)}<0\}}.
	\end{align*}
	Finally, the proof follows using equations \eqref{pp}, \eqref{5.9},  \eqref{5.10} and \eqref{5.11}.
\end{proof}
\noindent
To prove the main reliability estimate, we observe from Lemma \ref{reliability}, it is enough to provide an estimate in the maximum norm of $w^{(t)}$ in terms of the local error estimator terms defined in \eqref{est1}-\eqref{est2}-\eqref{est3}-\eqref{est4}. The technique which we used is motivated by the articles \cite{nochetto2003pointwise,demlow2012pointwise}. We perform the analysis using the Green's function of the unconstrained Poisson problem taking into account that Green's function is singular at $y^* \in \Omega$ where $w^{(t)}$ attains it's maximum. This helps us to improve the power of the logarithmic factor present in the resulting estimates.
\subsubsection{Bound on $||w^{(t)}||_{L_{\infty}(\Omega)}$}Let $y^* \in \Omega \setminus \partial\Omega$ be such that $|w^{(t)}(y^*)|=||w^{(t)}||_{L_{\infty}(\Omega)}$, then in the view of equations \eqref{eqq} and \eqref{FFF}, we have
\begin{align}
w^{(t)}(y^*) =\int_{\Omega} w^{(t)}(\zeta) \delta_{y^*}(\zeta) d \zeta &= a(w^{(t)}, G(y^*, \cdot))= \langle G_h^{(t)}, G \rangle. \label{winf}
\end{align}
where $G(y^*, \cdot)$ denotes the Green's function with singularity at $y^*$. Therefore, to derive the upper bound on $||w^{(t)}||_{L_{\infty}(\Omega)}$, we try to bound the term $\langle G_h^{(t)}, G \rangle$. Let $\cT_1 \in \cT_h$ be the patch of elements touching $T^*$ where $y^* \in T^*$. Further, assume $\cT_2$ be the set of elements touching $\cT_1$. To find the bound on $\langle G_h^{(t)}, G \rangle$, we used the disjoint decomposition of our mesh partition $\cT_h$, i.e., $(\cT_h \cap \cT_1) \cup (\cT_h \setminus \cT_1)$. Finally, from the definition of $G_h^{(t)}$, we have
\begin{align*}
 \langle G_h^{(t)}, G \rangle = a_h(u-\tilde{u}_h^{(t)},G)+\langle \tilde{\sigma}-B_h \sigma_h^{(t)}, G \rangle.
\end{align*}
Next, we bound the terms $\langle G_h^{(t)}, G \rangle$ for $t=1,2$ seperately (Case-I and Case-II) and the bound on $||w^{(t)}||_{L_{\infty}(\Omega)}$ follows using \eqref{winf}.
\subsubsection{Case-I (t=1)} \label{ssec1}
\begin{align*}
\langle G_h^{(1)}, G \rangle &= a_h(u-\tilde{u}_h^{(1)},G)+\langle \tilde{\sigma}-B_h \sigma_h^{(1)}, G \rangle \\ & = a_h(u-\tilde{u}_h^{(1)},G)+\langle \tilde{\sigma}- \sigma_h^{(1)}, G \rangle \hspace{1cm}\quad (\text{from definition of}~B_h \sigma_h^{(1)})  \\ &= a(u, G) + \langle \tilde{\sigma}, G \rangle - a_h(\tilde{u}_h^{(1)},G)- (\sigma_h^{(1)}, G) \quad (\text{from definition of}~a_h(\cdot, \cdot)) \\ &= (f, G) - a_h(\tilde{u}_h^{(1)},G)- (\sigma_h^{(1)}, G) \qquad (\text{from remark}~(\ref{extt})) \\ & = (f, G- G_c) + (f, G_c) -a_h(u_h^{(1)},G-G_c) -a_h(u_h^{(1)},G_c) \\ & \hspace{0.5cm} -a_h(\tilde{u}_h^{(1)}-u_h^{(1)},G) - (\sigma_h^{(1)}, G-G_c )- ( \sigma_h^{(1)}, G_c ) \\ & = \left((f, G- G_c)-a_h(u_h^{(1)},G-G_c)-( \sigma_h^{(1)}, G-G_c) \right) + (f, G_c)\\ & \hspace{0.5cm} -a_h(u_h^{(1)},G_c)-( \sigma_h^{(1)}, G_c )-a_h(\tilde{u}_h^{(1)}-u_h^{(1)},G). \quad (\text{rearranging terms}) \end{align*}
Adding and subtracting $b^*_h(u_h^{(1)},G_c)$, and performing integration by parts \cite{s1989topics} in the above equation, we get
\begin{align}
\langle G_h^{(1)}, G \rangle  &= \Bigg( \sum_{T \in \cT_h} \int_T (f+\Delta_h u_h^{(1)}-\sigma_h^{(1)})(G-G_c)~dx-\sum_{T \in \cT_h}\int_{\partial T}\frac{\partial u_h|_T}{\partial \eta} (G-G_c) \Bigg) \nonumber \\ & \hspace{0.5cm}+b^*_h(u_h^{(1)},G_c) + (f,G_c) - \mathcal{A}^*_h(u_h^{(1)},G_c)- ( \sigma_h^{(1)}, G_c )-a_h(\tilde{u}_h^{(1)}-u_h^{(1)},G) \nonumber
\end{align}
Using $\mathcal{A}_h(v_h,w_h)= \mathcal{A}^*_h (v_h, w_h) \quad \forall v_h, w_h \in V_h$, we obtain
\begin{align} \label{5.18}
\langle G_h^{(1)}, G \rangle &= \sum_{T \in \cT_h} \int_T (f+\Delta_h u_h^{(1)}-\sigma_h^{(1)})(G-G_c)~dx \nonumber \\ & \hspace{0.5cm} -\sum_{e \in \cE_h^i}\int_e\sjump{\nabla u_h^{(1)}}\{\{G-G_c\}\}~ds - \sum_{e \in \cE_h} \int_e \{\{\nabla u_h^{(1)}\}\} [[G-G_c]] ~ds  \nonumber \\ & \hspace{0.5cm}+b^*_h(u_h^{(1)},G_c) + (f,G_c) - \mathcal{A}_h(u_h^{(1)},G_c) - ( \sigma_h^{(1)}, G_c )-a_h(\tilde{u}_h^{(1)}-u_h^{(1)},G).
\end{align} 
For any $w \in V_h$, it holds $( \sigma_h^{(1)},w )= (f,w) - \mathcal{A}_h(u_h^{(1)},w)$  (recalling \eqref{sigma1}).

Finally, equation \eqref{5.18} reduces to
\begin{align} \label{eq1}
\langle G_h^{(1)}, G \rangle_{-1,1} &= \left(\sum_{T \in \cT_h} \int_T(f+\Delta_h u_h^{(1)}-\sigma_h^{(1)})( G-G_c)~ dx -\sum_{e \in \cE_h^i}\int_e\sjump{\nabla u_h^{(1)}}\{\{G-G_c\}\}~ds \right) \nonumber \\ & \hspace{0.5cm} +b^*_h(u_h^{(1)},G) -\left(b^*_h(u_h^{(1)},G-G_c)+ \sum_{e \in \cE_h} \int_e \{\{\nabla u_h^{(1)}\}\} [[G-G_c]]~ds\right)  \nonumber \\ & \hspace{0.5cm} -a_h(\tilde{u}_h^{(1)}-u_h^{(1)},G).
\end{align}
Next, we bound each terms on the right hand side of equation \eqref{eq1}.  Using H\"older's inequality, we obtain
\begin{enumerate}
	\item[a)] \begin{align*}
	\sum_{T \in \cT_h} \int_T (f+\Delta_h u_h^{(1)}-\sigma_h^{(1)})( G&-G_c)~ dx \\  & \leq \sum_{T \in \cT_h} \big|\big|h^2_T (f+\Delta_h u_h^{(1)}-\sigma_h^{(1)})\big|\big|_{L_{\infty}(T)} \big|\big|h^{-2}_T(G-G_c)\big|\big|_{L_{1}(T)}, \\& \lesssim \eta_1 \sum_{T \in \cT_h} \big|\big|h^{-2}_T(G-G_c)\big|\big|_{L_{1}(T)}.
	\end{align*}
 Using the approximation properties of the local projection operator $\Pi_h$ (recalling Lemma \ref{approx}), we obtain
\begin{align*} 
\sum_{T \in \cT_h} \int_T (f+\Delta_h u_h^{(1)}-\sigma_h^{(1)})&( G-G_c)~ dx \\ & \lesssim  \eta_1 \left(\sum_{T \in \cT_h \cap \cT_1}\big|\big|h^{-2}_T(G-G_c)\big|\big|_{L_{1}(T)} + \sum_{T \in \cT_h \setminus \cT_1} \big|\big|h^{-2}_T(G-G_c)\big|\big|_{L_{1}(T)} \right), \\ & \lesssim \eta_1 \left(\sum_{T \in \cT_h \cap \cT_1} h^{-1}_T\big|G\big|_{1,1,T} + \sum_{T \in \cT_h \setminus \cT_1} \big|G\big|_{2,1,T} \right).
\end{align*}
	\item[b)] Using H\"older's inequality, we find
	 \begin{align*}
-\sum_{e \in \cE_h^i}\int_e\sjump{\nabla u_h^{(1)}}\{\{G-G_c\}\}~ds &\leq \Bigg|\sum_{e \in \cE_h^i}\int_e\sjump{\nabla u_h^{(1)}}\{\{G-G_c\}\}~ds \Bigg|, \\& \leq \sum_{e \in \cE_h^i} \big|\big|h_e \sjump{\nabla u_h^{(1)}} \big|\big|_{L_{\infty}(e)} h_e^{-1} \big|\big|\{\{G-G_c\}\}\big|\big|_{L_{1}(e)}, \\ & \leq \eta_2  \sum_{e \in \cE_h^i}h_e^{-1} \big|\big|\{\{G-G_c\}\}\big|\big|_{L_{1}(e)}.
\end{align*} 
 A use of trace inequality \eqref{trace} yields
\begin{align*}
-\sum_{e \in \cE_h^i}\int_e\sjump{\nabla u_h^{(1)}}\{\{G-G_c\}\}~ds& \leq \eta_2 \left( \sum_{T \in \cT_h} \Big[h^{-2}_T \big|\big|G-G_c\big|\big|_{L_{1}(T)}+ h^{-1}_T \big|G-G_c\big|_{{1,1,T}} \Big] \right), \\& \leq \eta_2 \left( \sum_{T \in \cT_h \cap \cT_1} h^{-1}_T\big|G\big|_{1,1,T} + \sum_{T \in \cT_h \setminus \cT_1} \big|G\big|_{2,1,T}  \right).
\end{align*}
\item[c)] From equation \eqref{right}, we have \begin{align*}
b^*_h(u_h^{(1)},G-G_c)=- \sum_{e \in \cE_h} \int_e \{\{\nabla \Pi_h(u_h^{(1)}) \}\} [[G-G_c]]& ds -~ \theta \sum_{e \in \cE_h} \int_e \{\{\nabla \Pi_h(G-G_c)\}\} [[u_h^{(1)}]] ds  \\ & -\sum_{e \in \cE_h} \int_e \frac{\eta}{h_e}[[u_h^{(1)}]][[G-G_c]]~ds. \end{align*}
\noindent
We have $\Pi_h(G-G_c)=0$ and $\Pi_h(w)=w~~ \forall w \in V_h$ (remark \ref{proj}), hence using these fact, we have
\begin{align*} -b^*_h(u_h^{(1)},G-G_c) -& \sum_{e \in \cE_h} \int_e \{\{\nabla u_h^{(1)}\}\} [[G-G_c]]~ds
\leq \Bigg| \sum_{e \in \cE_h} \int_e \frac{\eta}{h_e}[[u_h^{(1)}]][[G-G_c]]~ds
\Bigg|,  \\ & \hspace{-1cm}\lesssim ||\sjump{u_h^{(1)}}||_{L_\infty({\cE_h})}
\sum_{e \in \cE_h} h_e^{-1} ||[[G-G_c]]||_{L_{1}(e)}, \\
& \hspace{-1cm}\lesssim ||\sjump{u_h^{(1)}}||_{L_\infty({\cE_h})} \Bigg( \sum_{T \in \cT_h} h^{-2}_T ||G-G_c||_{L_{1}(T)} + h^{-1}_T ||G-G_c||_{1,1,T} \Bigg), \hspace{0.3cm} (\text{using}~ \eqref{trace})  \\
&\hspace{-1cm}\lesssim \eta_3 \left( \sum_{T \in \cT_h \cap \cT_1} h_T^{-1} |G|_{1,1,T} + \sum_{T \in \cT_h \setminus \cT_1} |G|_{2,1,T} \right).
\end{align*}
\item[d)] We have,
\begin{align*}
-a_h(\tilde{u}_h^{(1)}-u_h^{(1)},G) & \leq \sum_{T \in \cT_h} \left| \int_T (\nabla \tilde{u}_h^{(1)}-u_h^{(1)}) \cdot \nabla G ~dx\right|, \\& \leq \sum_{T \in \cT_h} \big|\tilde{u}_h^{(1)}-u_h^{(1)}\big|_{1,\infty,T} \big|G\big|_{1,1,T},\quad \text{(using H\"older's inequality)} \\& \lesssim \eta_3 \sum_{T \in \cT_h} h_T^{-1}\big| G\big|_{1,1,T}.\quad \text{(using equation \eqref{ppp})} 
\end{align*}
\item[e)] Using $[[G]]=0$ on  $e \in \cE_h^i$, equations \eqref{right} and \eqref{trace}, inverse inequalitites (Lemma \ref{inverse}) and Lemma \ref{approx}, we obtain \begin{align*}
b^*_h(u_h^{(1)},G) & = \sum_{e \in \cE_h} \int_e \{\{\nabla G_c\}\} [[u_h^{(1)}]] ~ds \\ & \leq ||\sjump{u_h^{(1)}}||_{L_\infty({\cE_h})} \sum_{e \in \cE_h} ||\nabla G_c||_{L_{1}(e)}  \\ & \lesssim \eta_3 \left( \sum_{T \in \cT_h} h^{-1}_T |G_c|_{1,1,T}+ \sum_{T \in \cT_h} |G_c|_{2,1,T}\right)\hspace{0.3cm} \\ & = \eta_3 \left( \sum_{T \in \cT_h} h^{-1}_T |G_c|_{1,1,T}+ \sum_{T \in \cT_h \setminus \cT_1} |G_c|_{2,1,T} +\sum_{T \in \cT_h \cap \cT_1} h_T^{-1} |G_c|_{1,1,T} \right) \\ & \leq \eta_3 \left( \sum_{T \in \cT_h} h^{-1}_T |G|_{1,1,T}+ \sum_{T \in \cT_h \setminus \cT_1} |G|_{2,1,T}\right)
\end{align*}
\end{enumerate}

\noindent
Finally, the following bound holds
\begin{align} \label{Ghb1}
\langle G_h^{(1)}, G \rangle &\lesssim  (\eta_1 + \eta_2 + \eta_3) \left(\sum_{T \in \cT_h \cap \cT_1} h^{-1}_T\big|G\big|_{1,1,T} + \sum_{T \in \cT_h \setminus \cT_1} \big|G\big|_{2,1,T} \right) \notag \\ & \hspace{0.3cm}+ \eta_3 \sum_{T \in \cT_h} h^{-1}_T |G|_{1,1,T}.
\end{align}
Next, we bound the term $\langle G_h^{(t)}, G \rangle $ for $t=2$.
\subsubsection{Case-II (t=2)} \label{ssec2}
\begin{align*}
\langle G_h^{(2)}, G \rangle &= a_h(u-\tilde{u}_h^{(2)},G)+\langle \tilde{\sigma}-B_h \sigma_h^{(2)}, G \rangle \\ & = a_h(u-\tilde{u}_h^{(2)},G)+\langle \tilde{\sigma}- Q_h \sigma_h^{(2)}, G \rangle  \hspace{1cm}\quad (\text{from definition of}~B_h \sigma_h^{(2)})  \\ &= a(u, G) + \langle \tilde{\sigma}, G \rangle - a_h(\tilde{u}_h^{(2)},G)- \langle Q_h \sigma_h^{(2)}, G \rangle \\ &= (f, G) - a_h(\tilde{u}_h^{(2)},G)- \langle \sigma_h^{(2)}, G \rangle +  (\sigma_h^{(2)}- Q_h \sigma_h^{(2)}, G ) \hspace{0.2cm}\quad (\text{from remark }(\ref{extt})) \\ & = (f, G- G_c) + (f, G_c) -a_h(u_h^{(2)},G-G_c) -a_h(u_h^{(2)},G_c) \\ & \hspace{0.5cm} -a_h(\tilde{u}_h^{(2)}-u_h^{(2)},G) - \langle \sigma_h^{(2)}, G-G_c \rangle - \langle \sigma_h^{(2)}, G_c \rangle +( \sigma_h^{(2)}- Q_h \sigma_h^{(2)}, G) \\ & = \left((f, G- G_c)-a_h(u_h^{(2)},G-G_c)- \langle \sigma_h^{(2)}, G-G_c \rangle \right) + (f, G_c)\\ & \hspace{0.5cm} -a_h(u_h^{(2)},G_c)- \langle \sigma_h^{(2)}, G_c \rangle-a_h(\tilde{u}_h^{(2)}-u_h^{(2)},G) +( \sigma_h^{(2)}- Q_h \sigma_h^{(2)}, G) \\ &= \Big( (f+\Delta_h u_h^{(2)}-\sigma_h^{(2)}, G-G_c) -\sum_{e \in \cE_h^i}\int_e\sjump{\nabla u_h^{(2)}}\{\{G-G_c\}\}~ds \Big) \\ & \hspace{0.5cm} + (f,G_c) - \mathcal{A}^*_h(u_h^{(2)},G_c)- \langle \sigma_h^{(2)}, G_c \rangle +b_h^*(u_h^{(2)},G_c) -a_h(\tilde{u}_h^{(2)}-u_h^{(2)},G)\\ & \hspace{0.5cm} + ( \sigma_h^{(2)}- Q_h \sigma_h^{(2)}, G ) - \sum_{e \in \cE_h} \int_e \{\{\nabla u_h^{(2)}\}\} [[G-G_c]]~ds
\end{align*}
Note that $ \mathcal{A}^*_h(w,v)=\mathcal{A}_h(w,v) \quad \forall w, v \in V_h$ and for any $w \in V_h$, we have $\langle \sigma_h^{(2)}, w \rangle_h= (f,w)- \mathcal{A}_h(u_h^{(2)},w)$, finally we include these arguments to obtain
\begin{align} \label{eq2}
\langle G_h^{(2)}, G \rangle &= \Bigg( \sum_{T \in \cT_h} \int_T(f+\Delta_h u_h^{(2)}-\sigma_h^{(2)})( G-G_c)~ dx -\sum_{e \in \cE_h^i}\int_e\sjump{\nabla u_h^{(2)}}\{\{G-G_c\}\}~ds \Bigg) \nonumber \\ & \hspace{0.5cm} +\langle \sigma_h^{(2)}, G_c \rangle_h -\langle \sigma_h^{(2)}, G_c \rangle +b^*_h(u_h^{(2)},G)-a_h(\tilde{u}_h^{(2)}-u_h^{(2)},G) \nonumber\\ & \hspace{0.5cm} + ( \sigma_h^{(2)}- Q_h \sigma_h^{(2)}, G - Q_h G )  - \sum_{e \in \cE_h} \int_e \{\{\nabla u_h^{(2)}\}\} [[G-G_c]]~ds-b^*_h(u_h^{(2)},G-G_c) \nonumber\\ &= \Bigg( \sum_{T \in \cT_h} \int_T(f+\Delta_h u_h^{(2)}-\sigma_h^{(2)})( G-G_c)~ dx -\sum_{e \in \cE_h^i}\int_e\sjump{\nabla u_h^{(2)}}\{\{G-G_c\}\}~ds \Bigg) \nonumber \\ & \hspace{0.8cm}+
\underbrace{(\sigma_h^{(2)}- Q_h \sigma_h^{(2)} , \tilde{\Pi}_h^{-1}G_c-G_c)}_\text{Term 1} +b^*_h(u_h^{(2)},G)-a_h(\tilde{u}_h^{(2)}-u_h^{(2)},G) \nonumber\\ & \hspace{0.8cm} + \underbrace{ ( \sigma_h^{(2)}- Q_h \sigma_h^{(2)}, G - Q_h G )}_\text{Term 2}- \sum_{e \in \cE_h} \int_e \{\{\nabla u_h^{(2)}\}\} [[G-G_c]]~ds-b^*_h(u_h^{(2)},G-G_c).
\end{align}

 \noindent
The last inequality follows from the following fact $\langle \sigma_h^{(2)}, G_c \rangle_h= (\sigma_h^{(2)}, \tilde{\Pi}_h^{-1}G_c )$ (from the definition of $\langle \cdot, \cdot \rangle_h$) and $( Q_h \sigma_h^{(2)} , \tilde{\Pi}_h^{-1}G_c-G_c)=0.$ We note that all the terms are similar in equations \eqref{eq1} and \eqref{eq2} except for Term 1 and Term 2. Thereby, we need to bound only these two terms as the bound on the other terms will follow as in Case-I. In view of  equations \eqref{key} and \eqref{4.16}, we have
\begin{enumerate}
	\item[a)] \begin{align*}
(\sigma_h^{(2)}- Q_h \sigma_h^{(2)} ,\tilde{\Pi}_h^{-1}G_c-G_c) &\leq \sum_{T \in \cT_h} \big|\big|\sigma_h^{(2)}- Q_h \sigma_h^{(2)}\big|\big|_{L_{\infty}(T)} \big|\big|\tilde{\Pi}_h^{-1}G_c-G_c\big|\big|_{L_{1}(T)}, \\ & \lesssim \sum_{T \in \cT_h} h^2_T \big|\big|\sigma_h^{(2)}- Q_h \sigma_h^{(2)}\big|\big|_{L_{\infty}(T)} h^{-1}_T \big|G\big|_{1,1,T}, \\& \lesssim \eta_4  \sum_{T \in \cT_h} h^{-1}_T \big|G\big|_{1,1,T}.
	\end{align*}
	\item[b)] Using Lemma \ref{approximation}, we bound
\end{enumerate} \begin{align*}
\langle \sigma_h^{(2)}- Q_h \sigma_h^{(2)}, G - Q_h G \rangle \leq  \eta_4 \sum_{T \in \cT_h} h^{-1}_T \big|G\big|_{1,1,T}.
	\end{align*}
	Combining, we find
	\begin{align} \label{Ghb2}
\langle G_h^{(2)}, G \rangle &\lesssim  (\eta_1 + \eta_2 + \eta_3) \left(\sum_{T \in \cT_h \cap \cT_1} h^{-1}_T\big|G\big|_{1,1,T} + \sum_{T \in \cT_h \setminus \cT_1} \big|G\big|_{2,1,T} \right) \notag \\ & \hspace{0.3cm}+ (\eta_3+\eta_4) \sum_{T \in \cT_h} h^{-1}_T |G|_{1,1,T}.
\end{align}
Finally, we combine equations \eqref{eq1}, \eqref{eq2}, \eqref{Ghb1} and \eqref{Ghb2} to get the bound of the term $\langle G_h^{(t)}, G \rangle$ for both cases $t=1$ and $t=2$. In view of \eqref{winf}, we have
 \begin{align*}
 ||w^{(t)}||_{L_{\infty}(\Omega)} & =\langle G_h^{(t)}, G \rangle \\ & \lesssim \Bigg( \sum_{T \in \cT_h} \int_T(f+\Delta_h u_h^{(t)}-\sigma_h^{(t)})( G-G_c)~ dx -\sum_{e \in \cE_h^i}\int_e\sjump{\nabla u_h^{(t)}}\{\{G-G_c\}\}~ds \Bigg) \\ & \hspace{0.5cm}+ (t-1)((\sigma_h^{(t)}- B_h \sigma_h^{(t)} , \tilde{\Pi}^{-1}_hG_c-G_c)) +b^*_h(u_h^{(t)},G)-a_h(\tilde{u}_h^{(t)}-u_h^{(t)},G)\\ &  \hspace{0.5cm}+ (t-1)( ( \sigma_h^{(t)}- B_h \sigma_h^{(t)}, G - Q_h G ))\\ & \hspace{0.5cm}- \sum_{e \in \cE_h} \int_e \{\{\nabla u_h^{(t)}\}\} [[G-G_c]]~ds-b^*_h(u_h^{(t)},G-G_c), \\ &  \lesssim (\eta_1 + \eta_2) \left( \sum_{T \in \cT_h \cap \cT_1} h^{-1}_T|G|_{1,1,T} + \sum_{T \in \cT_h \setminus \cT_1} |G|_{2,1,T}  \right)\\ & \hspace{1cm}+ (\eta_3 +(t-1) \eta_4)  \sum_{T \in \cT_h} h^{-1}_T |G|_{1,1,T}.
\end{align*}
Using the bound on $||w^{(t)}||_{L_{\infty}(\Omega)} $ and \eqref{theorem}, the following bound on $||u-u_h^{(t)}||_{L_{\infty}(\Omega)}$ follows
\begin{align} \label{errr}
||u-u_h^{(t)}||_{L_{\infty}(\Omega)} & \lesssim \Delta^* \left( \sum_{T \in \cT_h \cap \cT_1} h^{-1}_T|G|_{1,1,T} + \sum_{T \in \cT_h \setminus \cT_1} |G|_{2,1,T} +\sum_{T \in \cT_h \setminus \cT_1} h^{-1}_T |G|_{1,1,T} \right) \nonumber \\ &\hspace{0.7cm} + \|(\chi-u_h^{(t)})^{+}\|_{L_{\infty}(\Omega)}+
\|(u_h^{(t)}-\chi)^+\|_{L_{\infty}\{B_h \sigma_h^{(t)}<0\}} + \eta_3,
\end{align}
where $\Delta^*=\eta_1 + \eta_2+\eta_3 + (t-1)\eta_4$. Let $p = \frac{d}{d-1}$ and $q=d$, using H\"older's inequality we have \begin{align} \label{er}
	\sum_{T \in \cT_h \setminus \cT_1} h^{-1}_T |G|_{1,1,T} \leq \sum_{T \in \cT_h \setminus \cT_1} |G|_{1,p,T}.
	\end{align}
Therefore, in view of equation (\ref{errr}) and (\ref{er}), we find
\begin{align} \label{rel}
||u-u_h^{(t)}||_{L_{\infty}(\Omega)} & \lesssim \Delta^* \left( \sum_{T \in \cT_h \cap \cT_1} h^{-1}_T|G|_{1,1,T} + \sum_{T \in \cT_h \setminus \cT_1} |G|_{2,1,T} +\sum_{T \in \cT_h \setminus \cT_1} |G|_{1,p,T} \right) \nonumber \\ &\hspace{0.7cm} + \|(\chi-u_h^{(t)})^{+}\|_{L_{\infty}(\Omega)}+
\|(u_h^{(t)}-\chi)^+\|_{L_{\infty}\{B_h \sigma_h^{(t)}<0\}}+  \eta_3.
\end{align} 
Finally, we have the following reliability estimate.
\begin{theorem} \label{main}
	Let $\Omega \subset \mathbb{R}^d$ and $u \in \cK$ solves the inequality \eqref{SSSS}. For $t=1,2$, let $u_h^{(t)} \in \cK_h^{(t)}$ satisfy \eqref{DDD} and \eqref{DDDD}, respectively. Then,
	\begin{align*}
	||u-u_h^{(t)}||_{L_{\infty}(\Omega)} & \lesssim \eta_h,
	\end{align*}
	where $\eta_h=|\mbox{log}(h_{min})|\big(\eta_1 + \eta_2+\eta_3 + (t-1)\eta_4\big) + \|(\chi-u_h^{(t)})^{+}\|_{L_{\infty}(\Omega)}+
	\|(u_h^{(t)}-\chi)^+\|_{L_{\infty}\{B_h \sigma_h^{(t)}<0\}}$.
\end{theorem}
\begin{proof}
	Due to the shape regularity assumption, we have the existence of the constants $A_1, A_2 >0$ with $A_2>A_1$ such that
	\begin{align*}
	B_1:=B(y^*, A_1h_1) &\subset \cT_1, \\B_2:=B(y^*, A_2h_1) &\supset \cT_2.
	\end{align*} where $B(x,r)$ is the ball centered at $x$ with radius $r$. Using the regularity estimates from Lemma \ref{Grn}, we have the following estimates
	\begin{align*}
	 \sum_{T \in \cT_h \cap \cT_1} h^{-1}_T|G|_{1,1,T} \lesssim  \sum_{T \in \cT_h \cap \cT_2} h^{-1}_T|G|_{1,1,T} \lesssim  y_1^{-1} |G|_{1,1,(\Omega \cap B_2)} \lesssim 1,
	\end{align*}
	\begin{align*}
	\sum_{T \in \cT_h \setminus \cT_1} |G|_{2,1,T} \lesssim |G|_{2,1,(\Omega \setminus B_1)} \lesssim |log(h_1)|,
	\end{align*}
	\begin{align*}
	\sum_{T \in \cT_h \setminus \cT_1}  |G|_{1,p,T} \lesssim |G|_{1,p,(\Omega \setminus B_1)} \lesssim |log(h_1)|.
	\end{align*}
	In the view of equation \eqref{rel} and combining the estimates on the Green's function, we get the desired result.
\end{proof}

\subsection{Upper Bound II} We derive pointwise estimates corresponding to the error in  Lagrange multipliers $\sigma$ and $\sigma_h^{(t)}$ for $t=1,2$. Let $A \subset \Omega$ be any given subset, we denote
\begin{align*}
\tau_A:=\{T \in \cT_h~|~T \subset A\}
\end{align*} to be the patch around the set $A$. Let us define the functional space
\begin{align}
\mathcal{H}_A:= W^{2,1}_0(A):= \{v \in W^{2,1}(A)~:~v=\nabla v \cdot n =0~ \text{on}~ \partial A\},
\end{align} together with the norm $||v||_{\mathcal{H}_A}:= |v|_{1,1,A}+ |v|_{2,1,A}.$ We set $\mathcal{H}:=\mathcal{H}_{\Omega}= W^{2,1}_0(\Omega).$ For any $\mathcal{L} \in \mathcal{H}^*_A$, we define the operator norm $||\mathcal{L}||_{-2,\infty,A}$ as
\begin{align} \label{opnorm}
||\mathcal{L}||_{-2,\infty,A}:= \sup \{\langle \mathcal{L}, v \rangle~:~ v \in \mathcal{H}_A~,~||v||_{\mathcal{H}_A} \leq 1\}.
\end{align}
In the next two lemmas, we collect some bounds which will play a key role in proving the main result of this subsection.
{
\begin{lemma} \label{5.7}
For $t=1,2$, let $u_h^{(t)} \in \cK^{(t)}_h$ be the approximate solution of $u$ and assume $w \in \mathcal{H}$, then
	\begin{align}
b^*_h(u_h^{(t)},w) \lesssim \eta_3 ||w||_{\mathcal{H}}.
	\end{align}
\end{lemma}
\begin{proof} We have,
	\begin{align*}
	b^*_h(u_h^{(t)},w) & = \sum_{e \in \cE_h} \int_e \{\{\nabla w_c\}\} [[u_h^{(t)}]] \quad \text{(using [[w]]=0)} \\ & \leq ||\sjump{u_h^{(t)}}||_{L_\infty({\cE_h})} \sum_{e \in \cE_h} ||\nabla w_c||_{L_{1}(e)} \quad \text{(using H\"older's inequality \cite{ciarlet2002finite})} \\ & \leq \eta_3 \left( \sum_{T \in \cT_h} h^{-1}_T |w_c|_{1,1,T}+ \sum_{T \in \cT_h} |w_c|_{2,1,T}\right)\hspace{0.3cm} (\text{using Lemma \ref{trace}})\\ & \leq \eta_3 \left( \sum_{T \in \cT_h}  h^{-1}_T |w|_{1,1,T}+ \sum_{T \in \cT_h} |w|_{2,1,T}\right) \hspace{0.2cm} (\text{using equation \eqref{DR1}}) \\ & \lesssim \eta_3 \left( \sum_{T \in \cT_h}  |w|_{1,\frac{d}{d-1},T}+ \sum_{T \in \cT_h} |w|_{2,1,T}\right) \hspace{0.2cm} (\text{using Cauchy Schwarz inequality for } q=\frac{d}{d-1}) \\ & \leq \eta_3 ||w||_{\mathcal{H}} \hspace{0.2cm} (\text{using Lemma \ref{poincare}}).
	\end{align*}
\end{proof} }
\begin{lemma} \label{5.8}
	It holds that
	\begin{align}
	a_h(\tilde{u}_h^{(t)}-u_h^{(t)},v) \lesssim \eta_3 |v|_{2,1,\Omega}
	\end{align} 
	for any $v \in  \mathcal{H}.$ 
\end{lemma}
\begin{proof}
	In the view of Lemma \ref{lem:average}, integration by parts and trace inequality, we get
	\begin{align*}
	a_h(\tilde{u}_h^{(t)}-u_h^{(t)},v) &= \sum_{T \in \cT_h} \int_T \nabla(\tilde{u}_h^{(t)}-u_h^{(t)}) \cdot \nabla v~dx, \\ & \lesssim  - \sum_{T \in \cT_h} \int_T (\tilde{u}_h^{(t)}-u_h^{(t)}) \Delta v ~dx + \sum_{T \in \cT_h} \int_{\partial T} (\tilde{u}_h^{(t)}-u_h^{(t)})  \frac{\partial v}{\partial n_T} ~ds ,\\ & \lesssim ||\tilde{u}_h^{(t)}-u_h^{(t)}||_{L_{\infty}(\Omega)} |v|_{2,1,\Omega} , \\ & \lesssim \eta_3 |v|_{2,1,\Omega}.
	\end{align*}
\end{proof}
 
\noindent
The upper bound on the Galerkin functional $G_h^{(t)}$ in the dual norm defined in \eqref{opnorm} is proved in the next lemma.
\begin{lemma} \label{bGalerkin}
	For $t=1,2$, it holds that
	\begin{align}
	||G_h^{(t)}||_{-2,\infty,\Omega} \lesssim \sum_{i=1}^3 \eta_i + (t-1) \eta_4.
	\end{align}
\end{lemma}
\begin{proof}
	Let $v \in \mathcal{H}$ and $v_h \in V_h$ be the approximation of $v$. Using the similar ideas for both $t=1,2$ as in the subsections (\ref{ssec1}) and (\ref{ssec2}) and by definition of $G_h^{(t)}$, we obtain
	\begin{align*}
	 \langle G_h^{(t)}, v \rangle &=a_h(u-\tilde{u}_h^{(t)},v)+\langle  \tilde{\sigma} -B_h \sigma_h^{(t)}, v\rangle  \\ &= (f,v) -a_h(\tilde{u}_h^{(t)},v) - \langle  B_h \sigma_h^{(t)}, v\rangle \\ & = (f-\sigma_h^{(t)},v-v_h) -a_h(u_h^{(t)},v-v_h) - \langle  B_h \sigma_h^{(t)}-\sigma_h^{(t)}, v-v_h \rangle \\ & \hspace{0.4cm}+ (f,v_h) -a_h(u_h^{(t)},v_h) - \langle  B_h \sigma_h^{(t)}, v_h \rangle -a_h(\tilde{u}_h^{(t)}-u_h^{(t)},v) \\ & \leq \Bigg( \sum_{T \in \cT_h} \int_T(f+\Delta_h u_h^{(t)}-\sigma_h^{(t)})( v-v_h)~ dx -\sum_{e \in \cE_h^i}\int_e\sjump{\nabla u_h^{(t)}}(v-v_h)~ds \Bigg) \\ & \hspace{0.5cm}+ (t-1)((\sigma_h^{(t)}- B_h \sigma_h^{(t)} , \tilde{\Pi}^{-1}_hv_h-v_h)) +b_h(u_h^{(t)},v)-a_h(\tilde{u}_h^{(t)}-u_h^{(t)},v)\\ &  \hspace{0.5cm}+ (t-1)( ( \sigma_h^{(t)}- B_h \sigma_h^{(t)},v - Q_h v )) \\ & \hspace{0.5cm}- \sum_{e \in \cE_h} \int_e \{\{\nabla u_h^{(t)}\}\} [[v-v_h]]~ds-b^*_h(u_h^{(t)},v-v_h)
	\end{align*}
	Using the H\"older's inequality, discrete trace inequality  \eqref{3.1}, Lemmas \ref{5.7}, \ref{5.8} and \ref{key} and equations \eqref{DR2}, \eqref{4.2} and \eqref{4.16}, we have the desired proof.
\end{proof}

\noindent
Next, we prove the reliability estimate for the term	$||\tilde{\sigma}-\sigma^{(t)}_h ||_{-2, \infty, \Omega}$.
\begin{theorem}
	Let $\tilde{\sigma}$ be as defined in equation \eqref{CCCC} and for $t=1,2$, let $\sigma^{(t)}_h$ be the discrete Lagrange multipliers defined in equations \eqref{DLM} and \eqref{DLLM}, respectively, then 
	\begin{align} \label{5.21}
	||\tilde{\sigma}-\sigma^{(t)}_h ||_{-2, \infty, \Omega} \lesssim \eta_h,
	\end{align}
	where $\eta_h$ is defined in the Thoerem \ref{main}.
	\begin{proof}
		Let $v \in \mathcal{H}$. Using equation (\ref{QWA}), we have
		\begin{align*}
	\langle  \tilde{\sigma} -\sigma^{(t)}_h	,v\rangle &= \langle  \tilde{\sigma} -B_h \sigma_h^{(t)}, v\rangle + \langle  B_h \sigma_h^{(t)}- \sigma^{(t)}_h, v\rangle \\ &=a_h(u-\tilde{u}_h^{(t)},v) -	\langle G_h^{(t)}, v \rangle + \langle  B_h \sigma_h^{(t)}- \sigma^{(t)}_h, v\rangle .
		\end{align*}
		Using the definition of $B_h \sigma_h^{(t)}$ and using $ (Q_h \sigma_h^{(2)}- \sigma^{(2)}_h, Q_h(v)) =0 $, we conclude
		\begin{align} \label{eqq1}
			\langle  \tilde{\sigma} -\sigma^{(t)}_h	,v\rangle = a_h(u-\tilde{u}_h^{(t)},v) - \langle G_h^{(t)}, v \rangle + (Q_h \sigma_h^{(2)}- \sigma^{(2)}_h, v - Q_h(v)).
		\end{align}
		We deal with the first term on the right hand side in the equation (\ref{eqq1}). Using integration by parts and the fact $u- \tilde{u}_h^{(t)} =0$ on $\partial \Omega$, we get
		\begin{align*}
		a_h(u-\tilde{u}_h^{(t)},v) & =  \int_{\Omega} \nabla (u- \tilde{u}_h^{(t)} ) \cdot \nabla v dx  \\ & = - \int_{\Omega} (u- \tilde{u}_h^{(t)} ) \cdot \Delta v dx.
		\end{align*}
		In the view of Lemma (\ref{lem:average}) and triangle inequality, we have
		\begin{align*}
		a_h(u-\tilde{u}_h^{(t)},v) &\lesssim \big(||u-u_h^{(t)}||_{L_{\infty}(\Omega)}  + \eta_3 \big) |v|_{W^{2,1}(\Omega)}.
		\end{align*}
		Using the bounds on $||u-u_h^{(t)}||_{L_{\infty}(\Omega)}$ (Theorem \ref{main}), $||G_h^{(t)}||_{-2,\infty,\Omega}$ (Lemma \ref{bGalerkin}) and Lemma \ref{approximation}, we obtain the estimate \eqref{5.21}.
	\end{proof}
\end{theorem}
\subsection{Results for Conforming Finite Element Method.} The conforming finite element for the problem \eqref{SSSS} is to find $u^{(t)}_{hc}$ such that
\begin{align} \label{SSS}
a(u^{(t)}_{hc},v^{(t)}_{hc}-u^{(t)}_{hc}) \leq (f,v_{hc}^{(t)}-u^{(t)}_{hc}) \quad \forall v_{hc}^{(t)} \in \cK_{hc}^{(t)}.
\end{align}
where
\begin{align*}
\bullet~~~~~~~ \cK_{hc}^{(1)}&= \{v_{hc} \in V_c: Q_h(v_{hc}) \geq Q_h(\chi)\}, \\
\bullet~~~~~~~ \cK_{hc}^{(2)}&=\{ v_{hc} \in V_c: v_{hc}(r) \geq \chi(r)~~\forall r \in   \cV_{\mathcal{Q}_T}, ~~ \forall~ T~ \in \cT_h \}.
\end{align*}
Hence, our theory will lead to the following reliability estimates which is comparable with the estimator presented in \cite{KP:2022:Obstacle}:
\begin{lemma}
	Let $u \in \cK$ and $u^{(t)}_{hc} \in \cK_{hc}^{(t)}$ be the solutions of equation \eqref{SSSS} and \eqref{SSS}, respectively. Then, 
	\begin{align}
	||u-u_{hc}^{(t)}||_{L_{\infty}(\Omega)} \lesssim \eta_{hc},
	\end{align}
	where $\eta_{hc}=|log~h_{min}|(\eta_1 + \eta_2 + (t-1)\eta_4) + \|(\chi-u^{(t)}_{hc})^{+}\|_{L_{\infty}(\Omega)} +
	\|(u^{(t)}_{hc}-\chi)^+\|_{L_{\infty}\{B_h \sigma_h^{(t)}<0\}}.$
\end{lemma}
\subsection{Pointwise Lower Bound} After providing the reliable error estimator $\eta_h$, the question of whether or not this estimator would overestimate the error arises. This section aims to look into this problem regarding the estimator $\eta_h$. We say an estimator, 'locally efficient' if it is dominated by the error and the local data oscillations. Local efficiency prevents adaptive refinement against over-refinement. Standard bubble functions technique is used to prove the efficiency estimates of the error estimator in this section. We refer \cite{nochetto2003pointwise,nochetto2005fully,nochetto2006pointwise,demlow2016maximum} for the efficiency results for non linear problems in the supremum norm. Because of the quadratic nature of the discrete solution, the efficiency of the term	$	\|(u_h^{(t)}-\chi)^+\|_{L_{\infty}\{B_h \sigma_h^{(t)}<0\}} $ is still not clear.
We now state and prove the main results of this section.
\begin{lemma}
		 It holds that
		 \begin{align} \label{eq19}
		&	h_T^2 \|\Delta u_h^{(t)}+f-\sigma_h^{(t)}\|_{L_{\infty}(T)} + \|(\chi-u_h^{(t)})^{+}\|_{T} \lesssim \Big( \|u-u_h^{(t)}\|_{L_{\infty}(T)} + \|\sigma-\sigma_h^{(t)}\|_{-2,\infty,T} \nonumber \\ &\hspace{8cm}+ Osc(f,T)  \Big), \quad \forall~ T \in \cT_h.
		\end{align}
\end{lemma}
\begin{proof}
		 In the view of $u \geq \chi$, we find $(\chi-u_h^{(t)})^{+} \leq (u-u_h^{(t)})^{+}$, hence the bound on the second term of left hand side in the estimate follows immediately. Next, we estimate the volume residual $h_T^2 \|\Delta u_h^{(t)}+f-\sigma_h^{(t)}\|_{L_{\infty}(T)}$. Let $b_T \in P_{2(d+1)}(T)$ be an element bubble function which is zero on $\partial T$ and assumes unit value at the barycenter of T. Moreover, it holds $\| b_T\|_{L_{\infty}(T)} \leq C$, where $C$ is a positive constant. Let $\bar{f} \in P_0(T)$ be a piecewise constant approximation of $f$. Let $q^{(t)}$ be the extension of $q_T^{(t)}=b_T (\Delta u_h^{(t)}+\bar{f}-\sigma_h^{(t)})$ by zero to $\Omega.$ A use of inverse inequality (Lemma \ref{inverse}) yields the following 
		\begin{align} \label{eq11}
		h_T^2 \|\Delta u_h^{(t)}+\bar{f}-\sigma_h^{(t)}\|_{L_{\infty}(T)} & \lesssim h_T^{2-\frac{d}{2}} \|\Delta u_h^{(t)}+\bar{f}-\sigma_h^{(t)}\|_{L_2(T)} \nonumber \\
			h_T^4 \|\Delta u_h^{(t)}+\bar{f}-\sigma_h^{(t)}\|_{L_{\infty}(T)}^2 & \lesssim h_T^{4-d} \|\Delta u_h^{(t)}+\bar{f}-\sigma_h^{(t)}\|_{L_2(T)}^2 \quad\text{(squaring both sides)},
		\end{align} and due to the equivalence of norms in finite dimensional normed spaces and scaling arguments, there exist a constant $C >0$ such that
		\begin{align*}
		h_T^{4-d} \|\Delta u_h^{(t)}+\bar{f}-\sigma_h^{(t)}\|_{L_2(T)}^2 \leq C h_T^{4-d} \int_T  (\Delta u_h^{(t)}+\bar{f}-\sigma_h^{(t)}) q_T^{(t)} ~dx.
		\end{align*}
		Now, a use of integration by parts, equation \eqref{CCC} and noting that $q^{(t)} \in H^1_0(\Omega)$ yields
		\begin{align} \label{eq12}
		 h_T^{4-d} \int_T  (\Delta u_h^{(t)}+\bar{f}-\sigma_h^{(t)}) q_T^{(t)} ~dx &=  h_T^{4-d} \Bigg( \int_T  (f+ \Delta u_h^{(t)}-\sigma_h^{(t)}) q_T^{(t)} ~dx + \int_T (\bar{f}-f) q_T^{(t)}  ~dx \Bigg) \nonumber \\
		&= h_T^{4-d} \Bigg( \langle \sigma-\sigma_h^{(t)}, q^{(t)}\rangle+a(u-u_h^{(t)},q^{(t)})  \\ & \hspace{3cm}+ \int_T (-\bar{f}+f) q_T^{(t)}  ~dx \Bigg) \nonumber \\
		& \lesssim h_T^{4-d} \bigg(||\sigma-\sigma_h^{(t)}||_{-2,\infty,T} + \|u-u_h^{(t)}\|_{L_{\infty}(T)}\bigg) |D^2q_T^{(t)}|_{L_1(T)} \nonumber\\ & \quad+h_T^{4-d}||(\bar{f}-f)||_{L_{\infty}(T)}  ||q_T^{(t)}||_{L^1(T)}.
		\end{align}
	Next, using equations \eqref{eq11} and \eqref{eq12}, we have
	\begin{align*}
	h_T^4 \|\Delta u_h^{(t)}+\bar{f}-\sigma_h^{(t)}\|_{L_{\infty}(T)}^2 & \leq C h_T^{4-d}\Big(||\sigma-\sigma_h^{(t)}||_{-2,\infty,T} + \|u-u_h^{(t)}\|_{L_{\infty}(T)}\Big) |D^2q_T^{(t)}|_{L_1(T)}, \\ & \quad+ h_T^{4-d}||(\bar{f}-f)||_{L_{\infty}(T)}  ||q_T^{(t)}||_{L^1(T)}.
	\end{align*}
	A use of Lemma \ref{inverse} and the structure of $q_T^{(t)}$ yields
		\begin{align*}
	||q_T^{(t)}||_{L^1(T)} & \lesssim h_T^d ||q_T^{(t)}||_{L_{\infty}(T)}   \lesssim h_T^d ||(\Delta u_h^{(t)}+\bar{f}-\sigma_h^{(t)})||_{L_{\infty}(T)}
	\end{align*} and
	\begin{align*}
     |D^2q_T^{(t)}|_{L_1(T)} & \lesssim h^{-2}_T ||q_T^{(t)}||_{L^1(T)}  \lesssim h^{d-2}_T||q_T^{(t)}||_{L_{\infty}(T)} \\  & \lesssim h^{d-2}_T||(\Delta u_h^{(t)}+\bar{f}-\sigma_h^{(t)})||_{L_{\infty}(T)}.
	\end{align*}
	Therefore, combining the estimates, we derive the estimate \eqref{eq19}.
\end{proof}
\begin{lemma}
		The following estimate holds
		\begin{align*}
		\|\sjump{ u_h^{(t)}}\|_{L_{\infty}(e)} \lesssim \|u-u_h^{(t)}\|_{L_{\infty}(\Omega_e)} \quad \forall e \in \cE_h,
		\end{align*}
		where $\Omega_e$ is the union of elements sharing the face $e$.
\end{lemma}
\begin{proof}
	Let $e \in \cE^i_h$ and using the continuity of $u$, we have $[[u]]=0$ on $e$. Moreover, we obtain
	\begin{align*}
	\|\sjump{ u_h^{(t)}}\|_{L_{\infty}(e)} &= \|\sjump{u- u_h^{(t)}}\|_{L_{\infty}(e)} \\ & \lesssim \|u-u_h^{(t)}\|_{L_{\infty}(\Omega_e)},
	\end{align*}
	where $\Omega_e= \bar{T}_1 \cup \bar{T}_2$ and $\bar{T}_1,\bar{T}_2$ are elements sharing an edge/face $e$. For $e \in \cE^b_h$, the proof follows immediately as $u=0$ on $\partial \Omega$.
\end{proof}
\begin{lemma}
	 It holds that
	\begin{align*}
	h_e\|\sjump{\nabla u_h^{(t)}}\|_{L_{\infty}(e)} \lesssim \|u-u_h^{(t)}\|_{L_{\infty}(\Omega_e)}+ \|\sigma-\sigma_h^{(t)}\|_{-2,\infty,\Omega_e}+ Osc(f,\Omega_e) \quad  \forall~ e \in \cE_h^i.
	\end{align*}
\end{lemma}
\begin{proof}
		Let $b_e$ be the edge/face bubble function which is quadratic and continuous in $\Omega$ and assumes unit value at the center of edge/face $e$ and moreover, it satisfies $\| b_e\|_{L_{\infty}(\Omega_e)} \leq C.$ Let $q_e^{(t)} = \sjump{\nabla u_h^{(t)}}b_e$ on $\Omega_e$  where $\Omega_e=\bar{T}_{+} \cup \bar{T}_{-}$ with $T_{+}$ and $T_{-}$ are common elements corresponding to edge/face $e$. Next, we assume $q^{(t)}$ be the extension of $q_e^{(t)}$ by zero to $\Omega$ and observe $q^{(t)} \in H^1_0(\Omega)$. A use of Lemma \ref{inverse} provides
		\begin{align} \label{eq16}
			h_e\|\sjump{\nabla u_h^{(t)}}\|_{L_{\infty}(e)} & \lesssim h_e^{\frac{3-d}{2}} \|\sjump{\nabla u_h^{(t)}}\|_{L_{2}(e)}, \nonumber\\ 	h_e^2\|\sjump{\nabla u_h^{(t)}}\|_{L_{\infty}(e)}^2 & \lesssim h_e^{3-d} \|\sjump{\nabla u_h^{(t)}}\|_{L_{2}(e)}^2 \quad \text{(squaring both sides)} 
		\end{align}and due to the equivalence of norms in finite dimensional normed spaces we have an existence of a positive constant $C$ such that
		\begin{align} \label{eq17}
		h_e^{3-d} \|\sjump{\nabla u_h^{(t)}}\|_{L_{2}(e)}^2  \leq Ch_e^{3-d}  \int_e  \sjump{\nabla u_h^{(t)}} q_e^{(t)} ~ds.
		\end{align}
		A use of \eqref{CCC} and inverse inequalities (Lemma \ref{inverse}) yield
		\begin{align} \label{eq13}
		h_e^{3-d}  \int_e  \sjump{\nabla u_h^{(t)}} q_e^{(t)} ~ds &=  h_e^{3-d} \Bigg(\int_{\Omega_e}  {\nabla u_h^{(t)}}\cdot \nabla q_e^{(t)} ~dx +\int_{\Omega_e} \Delta u_h^{(t)} q_e^{(t)}  ~dx \Bigg) \nonumber \\
		&= h_e^{3-d} \Bigg(  (f,q^{(t)}) -a(u,q^{(t)})-\langle \sigma,q^{(t)} \rangle+\int_{\Omega_e}  {\nabla u_h^{(t)}}\cdot \nabla\phi_e ~dx \nonumber \\ & \hspace{3cm}+\int_{\Omega_e} \Delta u_h^{(t)} q_e^{(t)}  ~dx \Bigg)  \nonumber \\
		&= h_e^{3-d} \Bigg( \langle\sigma_h^{(t)}-\sigma,q^{(t)} \rangle +\int_{\Omega_e}  \nabla (u_h^{(t)}-u)\cdot \nabla q_e^{(t)} ~dx \nonumber \\ & \hspace{3cm }+\int_{\Omega_e} (\Delta u_h^{(t)}+f-\sigma_h^{(t)}) q_e^{(t)}  ~dx \Bigg)\nonumber  \\
		&\lesssim h_e^{3-d} \Bigg( \Big(\|\sigma-\sigma_h^{(t)}||_{-2,\infty,\Omega_e} \| + \|u-u_h^{(t)}\|_{L_{\infty}(\Omega_e)} \Big) |D^2(q_e^{(t)})|_{L_1(\Omega_e)} \nonumber \\& \quad+  \|({f}+\Delta u_h^{(t)}-\sigma_h^{(t)}) \|_{L_{\infty}(\omega_e)} \|q_e^{(t)}\|_{L^1(\Omega_e)}\Bigg). 
		\end{align}
	Using the structure of $q^{(t)}_e$ and Lemma \ref{inverse}, we get
	\begin{align} \label{eq14}
	|D^2(q_e^{(t)})|_{L_1(\Omega_e)}& \lesssim h^{-2}_e ||q_e^{(t)}||_{L^1(\Omega_e)} \lesssim h^{d-2}_e||q_e^{(t)}||_{L_{\infty}(\Omega_e)}, \nonumber \\  &\lesssim  h^{d-2}_e|| \sjump{\nabla u_h^{(t)}}||_{L_{\infty}(\Omega_e)}
	\end{align} and
	\begin{align} \label{eq15}
	\|q_e^{(t)}\|_{L^1(\Omega_e)} & \lesssim h_e^d ||q_e^{(t)}||_{L_{\infty}(\Omega_e)} \lesssim h_e^d ||\sjump{\nabla u_h^{(t)}}||_{L_{\infty}(\Omega_e)}.
	\end{align}We have the following estimate using equations \eqref{eq13}, \eqref{eq14} and \eqref{eq15}
	\begin{align} \label{eq18}
	h_e^{3-d}  \int_e  \sjump{\nabla u_h^{(t)}} q_e^{(t)} ~ds &  \lesssim h_e^{3-d} \Bigg( h_e^{d-2}\bigg(\|\sigma-\sigma_h^{(t)}||_{-2,\infty,\Omega_e} \| + \|u-u_h^{(t)}\|_{L_{\infty}(\Omega_e)} \bigg) \nonumber \\& \quad+ h_e^d \|({f}+\Delta u_h^{(t)}-\sigma_h^{(t)}) \|_{L_{\infty}(\omega_e)} \Bigg)||\sjump{\nabla u_h^{(t)}}||_{L_{\infty}(\Omega_e)}. 
	\end{align}Finally, we infer the desired estimate by using equations \eqref{eq16}, \eqref{eq17}, \eqref{eq18} and \eqref{eq19}.
\end{proof}
\begin{lemma}
For $T\in \cT_h$, it holds that
\begin{equation*}
h_T^2||\sigma_h^{(2)}-Q_h\sigma_h^{(2)}||_{L_{\infty}(T)} \lesssim h_T^2 \|\Delta u_h^{(2)}+f-\sigma_h^{(2)}\|_{L_{\infty}(T)}+Osc(f,{T}).
\end{equation*}
	\end{lemma}
		\begin{proof}
		Using triangle inequality and simple calculations, we obtain
			\begin{align*}
			||\sigma_h^{(2)}-Q_h\sigma_h^{(2)}||_{L_{\infty}(T)} & =||\sigma_h^{(2)}+(\Delta u_h^{(2)}-\Delta u_h^{(2)})+(f-f)+(Q_h f-Q_h f)-Q_h\sigma_h^{(2)}||_{L_{\infty}(T)}, \\&\leq \|\Delta u_h^{(2)}+f-\sigma_h^{(2)}\|_{L_{\infty}(T)}+\|Q_h\sigma_h^{(2)}-\Delta u_h^{(2)}-Q_h f\|_{L_{\infty}(T)} \\ & \hspace{2cm} +\|f-Q_h f\|_{L_{\infty}(T)}, \\
			&=\|\Delta u_h^{(2)}+f-\sigma_h^{(2)}\|_{L_{\infty}(T)}++\|Q_h\sigma_h^{(2)}-Q_h\Delta u_h^{(2)}-Q_h f\|_{L_{\infty}(T)}\\ & \hspace{2cm}+\|f-Q_h f\|_{L_{\infty}(T)}, \\
			&=\|\Delta u_h^{(2)}+f-\sigma_h^{(2)}\|_{L_{\infty}(T)}+ \|Q_h({\sigma_h^{(2)}-\Delta u_h^{(2)}-
				f})\,\|_{L_{\infty}(T)} \\ & \hspace{2cm}+\|f-Q_h f\|_{L_{\infty}(T)}.
			\end{align*}
			Finally, we have the desired result
			\begin{align} \label{5.43}
			h_T^2||\sigma_h^{(2)}-Q_h\sigma_h^{(2)}||_{L_{\infty}(T)} \lesssim h_T^2 \|\Delta u_h^{(2)}+f-\sigma_h^{(2)}\|_{L_{\infty}(T)}+Osc(f,{T}).
			\end{align}
			Hence, the efficiency of the term $\eta_4 = \underset{T\in \mathcal{T}_h}{max} ~h_T^2||\sigma_h^{(2)}-Q_h\sigma_h^{(2)}||_{L_{\infty}(T)}$ follows from \eqref{eq19} and \eqref{5.43}.
		\end{proof}

\section{Numerical Experiments} \label{sec6} 
\noindent
In this section, we implemented our error estimator $\eta_h$ (defined in the Theorem \ref{main}) to different obstacle problems and demonstrate it's performance. The following standard adaptive algorithm is used for mesh refineme nt.
\begin{equation*}
{\bf SOLVE}\rightarrow  {\bf ESTIMATE} \rightarrow {\bf
	MARK}\rightarrow {\bf REFINE}
\end{equation*}
The discrete nonlinear problems (equations   \eqref{DDD} and \eqref{DDDD}) are solved using the primal-dual active set algorithm  \cite{hintermuller2002primal} in the step SOLVE. In the ESTIMATE step, we compute our proposed a posteriori error estimator $\eta_h$ on each $T \in \cT_h$ and later in the MARK step, we employ the maximum marking strategy \cite{verfurth1996review} with parameter $\Gamma=0.4$ which seems appropriate for the error control in the supremum norm. Finally, we refine the adaptive mesh and obtain a new mesh using the newest vertex bisection algorithm \cite{verfurth1996review}. In our article, we consider two DG formulations: SIPG and NIPG \cite{gudi2014posteriori} (Remark \ref{DGform}). We choose the penalty parameter $\eta=45$ for the SIPG and $\eta=20$ for the NIPG method. 

We discuss and present below numerical results for the quadratic $(\mathbb{P}_2)$ elements with both (integral and quadrature points) constraints and linear $(\mathbb{P}_1)$ elements with integral constraints in two dimensions.
	\begin{figure}
	\begin{subfigure}[b]{0.4\textwidth}
		\includegraphics[width=\linewidth]{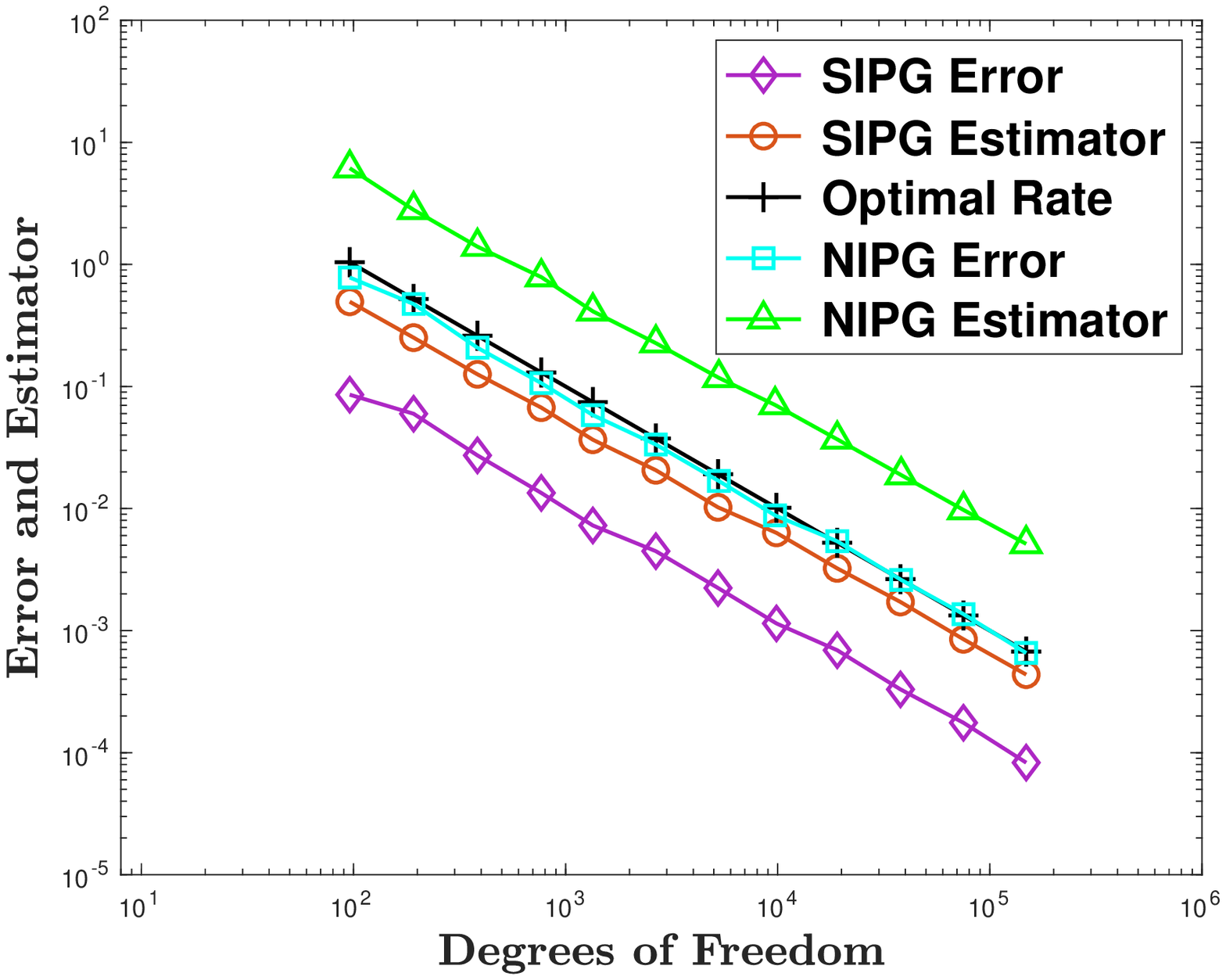}
		\caption{Error and Estimator}
		\label{fig:fig1}
	\end{subfigure}
	\begin{subfigure}[b]{0.4\textwidth}
		\includegraphics[width=\linewidth]{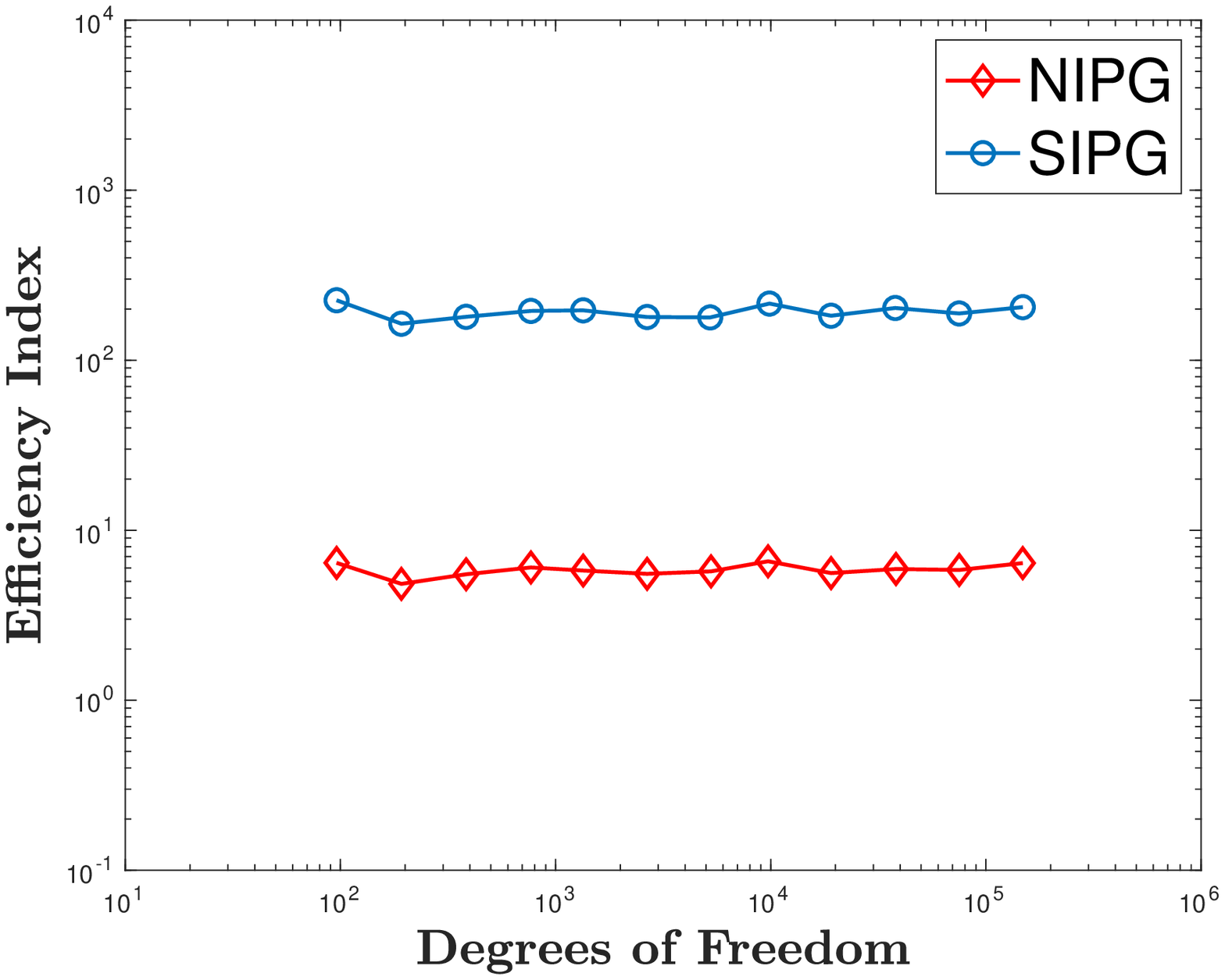}
		\caption{Efficiency Index}
		\label{fig:fig2}
	\end{subfigure}
	\caption{Error, Estimator and Efficiency Index of SIPG and NIPG methods ($\mathbb{P}_1$- Integral Constraints)  for Example \ref{ex1}. }\label{Fig1}
\end{figure} 
	
\begin{figure}
	\begin{subfigure}[b]{0.4\textwidth}
		\includegraphics[width=\linewidth]{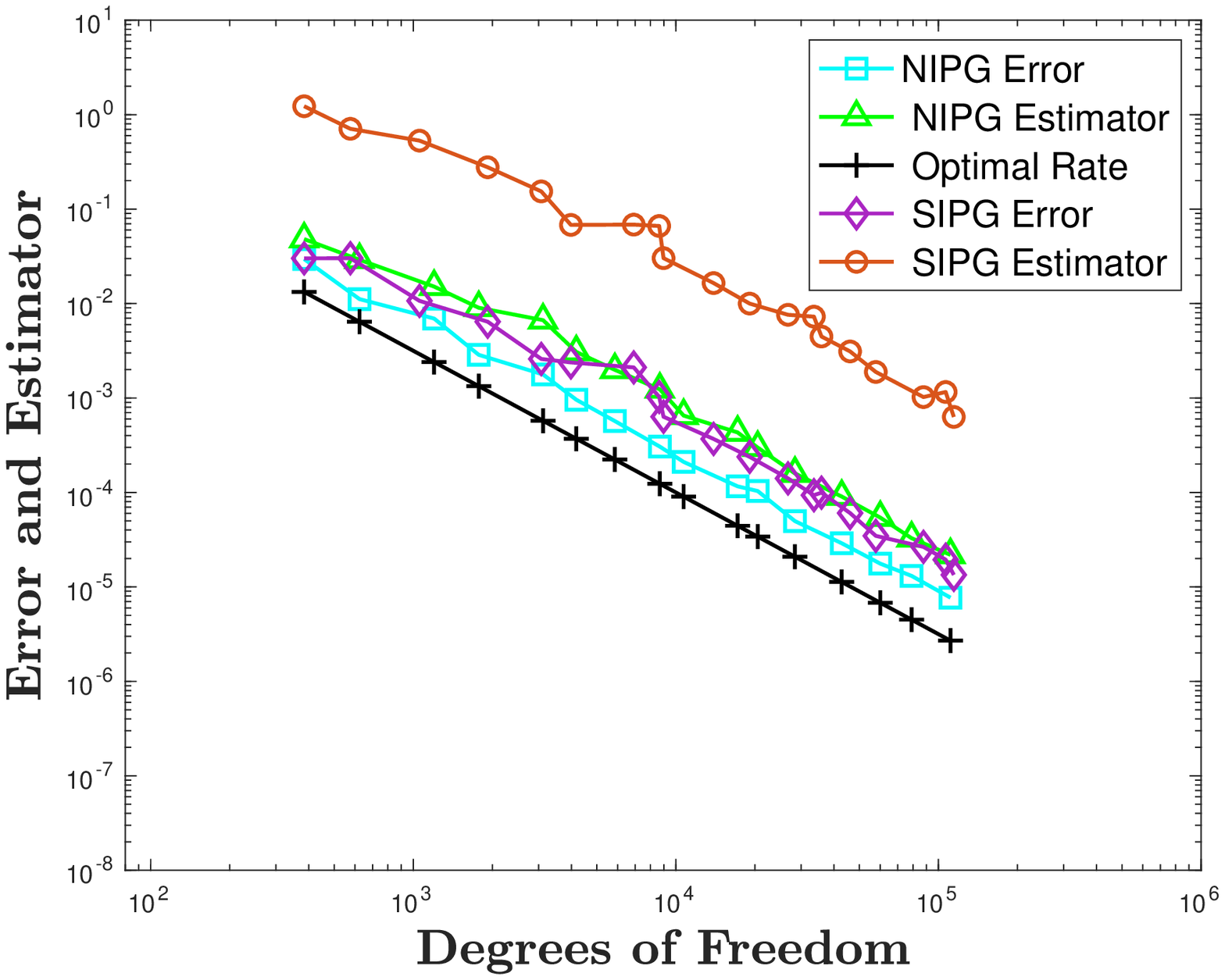}
			\caption{Error and Estimator}
		\label{fig:fig5}
	\end{subfigure}
	\begin{subfigure}[b]{0.4\textwidth}
		\includegraphics[width=\linewidth]{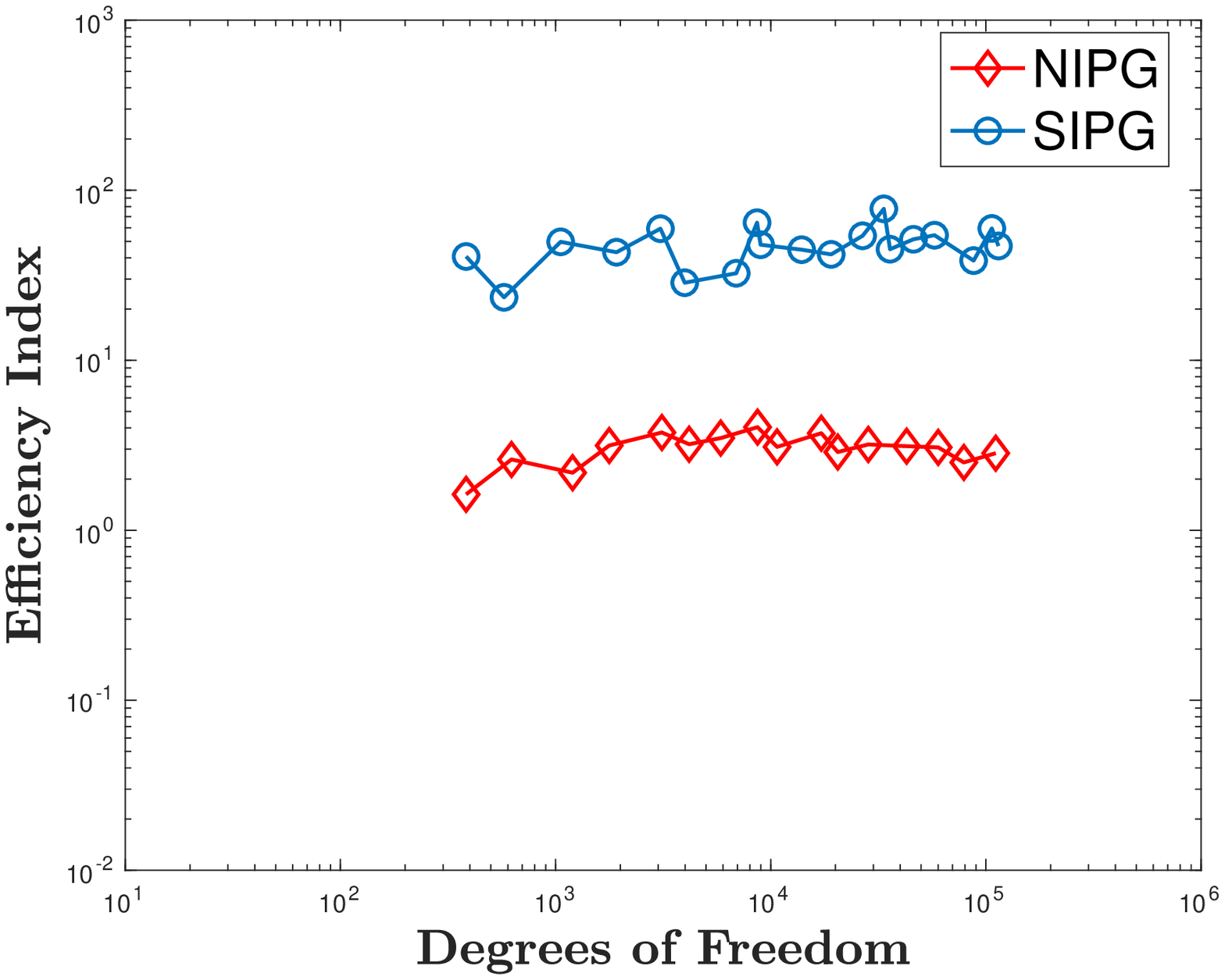}
			\caption{Efficiency Index}
		\label{fig:fig6}
	\end{subfigure}
	\caption{Error, Estimator and Efficiency Index of SIPG and NIPG methods ($\mathbb{P}_2$- Integral Constraints) for Example \ref{ex1}.}
	\label{Fig3}
\end{figure} 
\begin{figure}
	\begin{subfigure}[b]{0.4\textwidth}
		\includegraphics[width=\linewidth]{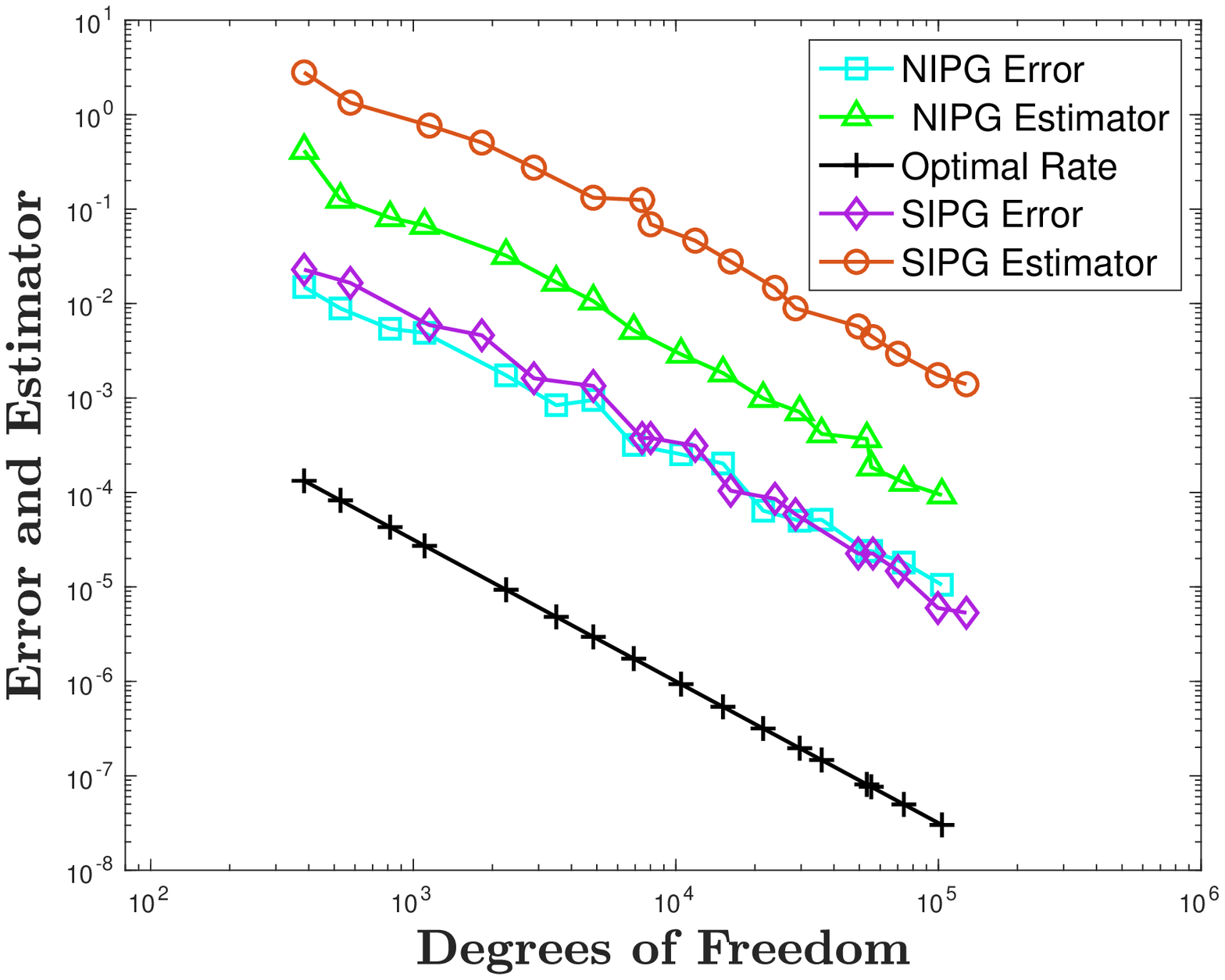}
			\caption{Error and Estimator}
		\label{fig:fig7}
	\end{subfigure}
	\begin{subfigure}[b]{0.4\textwidth}
		\includegraphics[width=\linewidth]{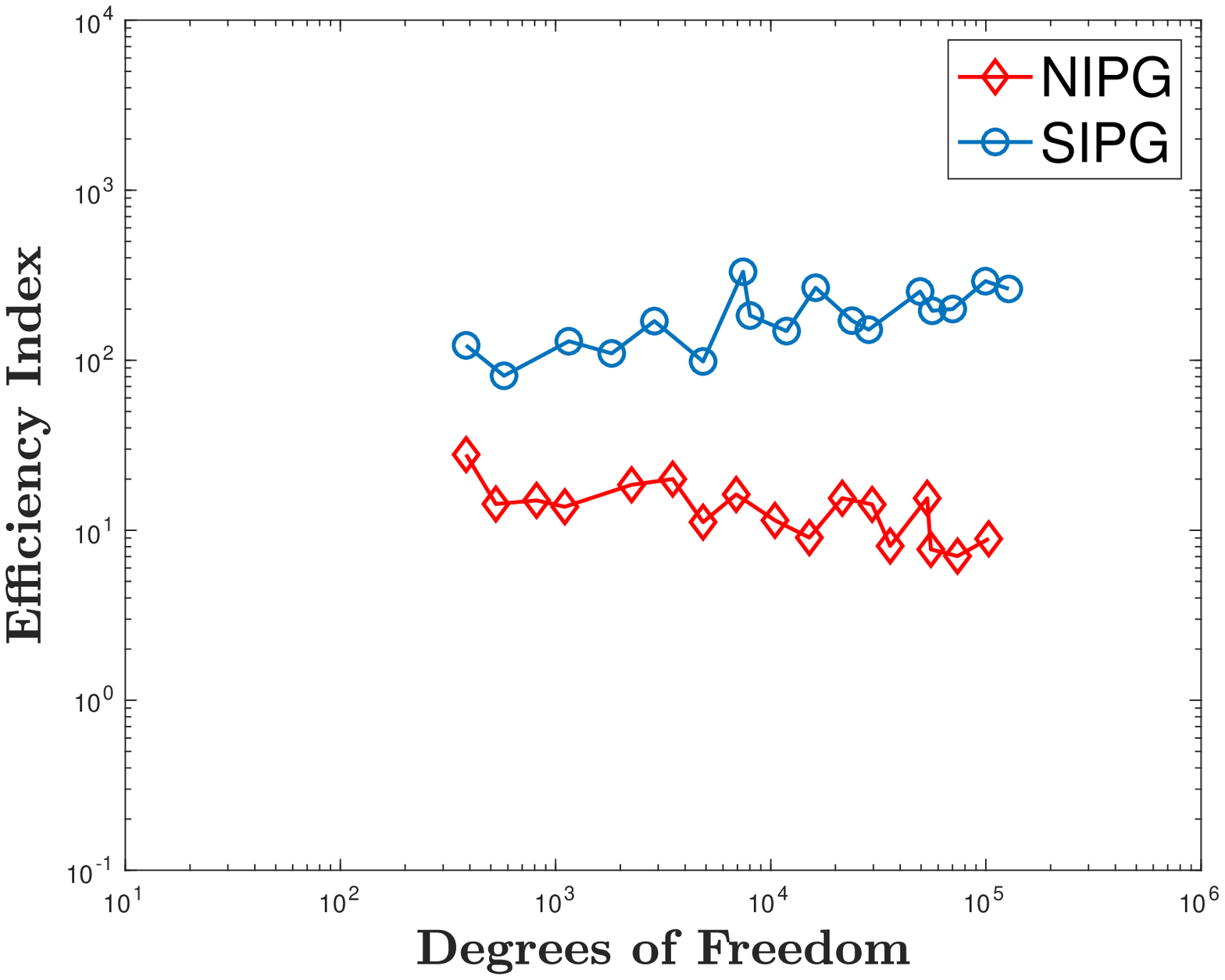}
			\caption{Efficiency Index}
		\label{fig:fig8}
	\end{subfigure}
	\caption{Error, Estimator and Efficiency Index of SIPG and NIPG methods ($\mathbb{P}_2$- Quadrature point Constraints) for Example \ref{ex1}. }\label{Fig4}
\end{figure} 
\begin{figure}
	\begin{subfigure}[b]{0.4\textwidth}
		\includegraphics[width=\linewidth]{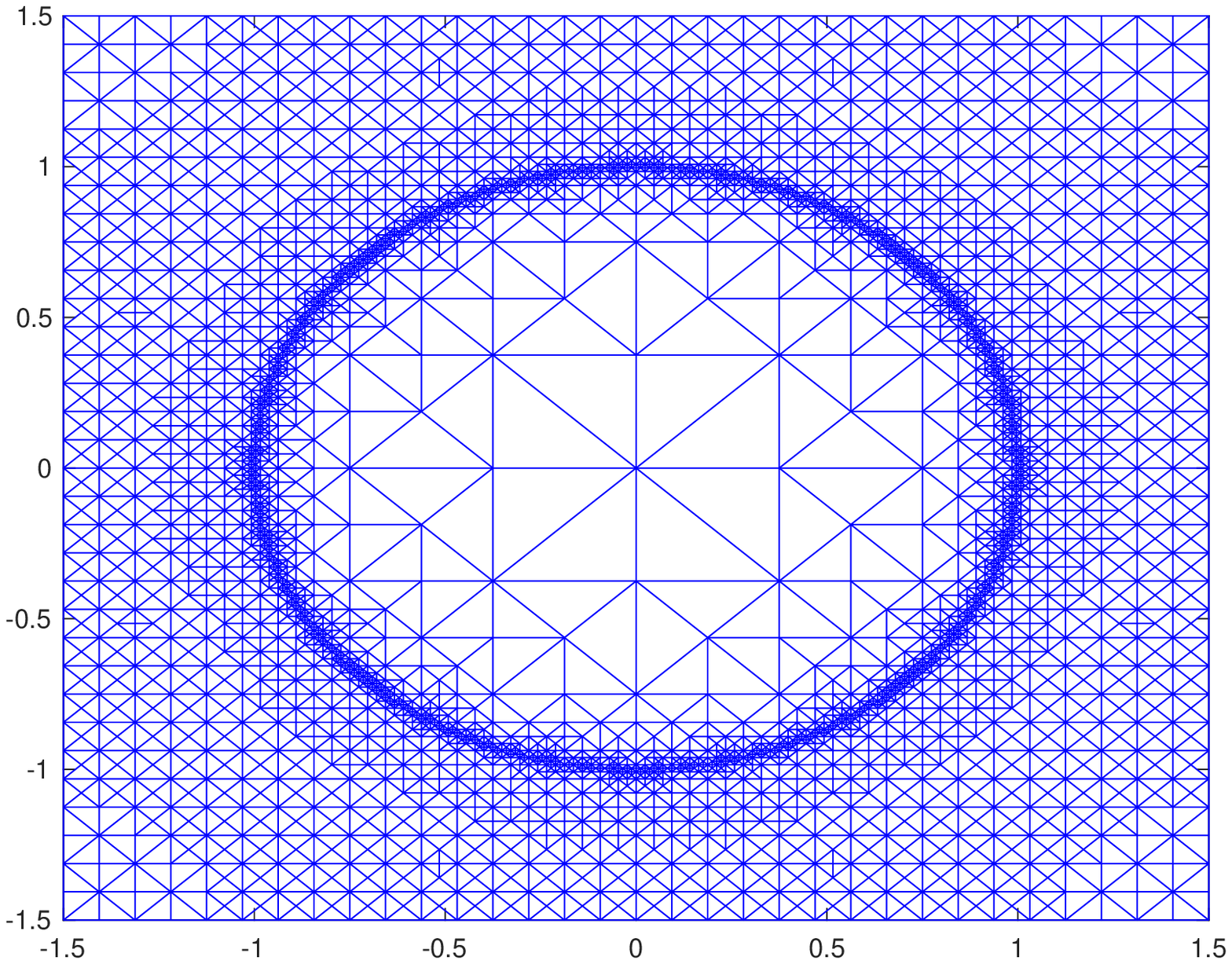}
		\caption{$\mathbb{P}_2$ (Quadrature point Constraints)}
		\label{fig:fig4}
	\end{subfigure}%
	\begin{subfigure}[b]{0.4\textwidth}
		\includegraphics[width=\linewidth]{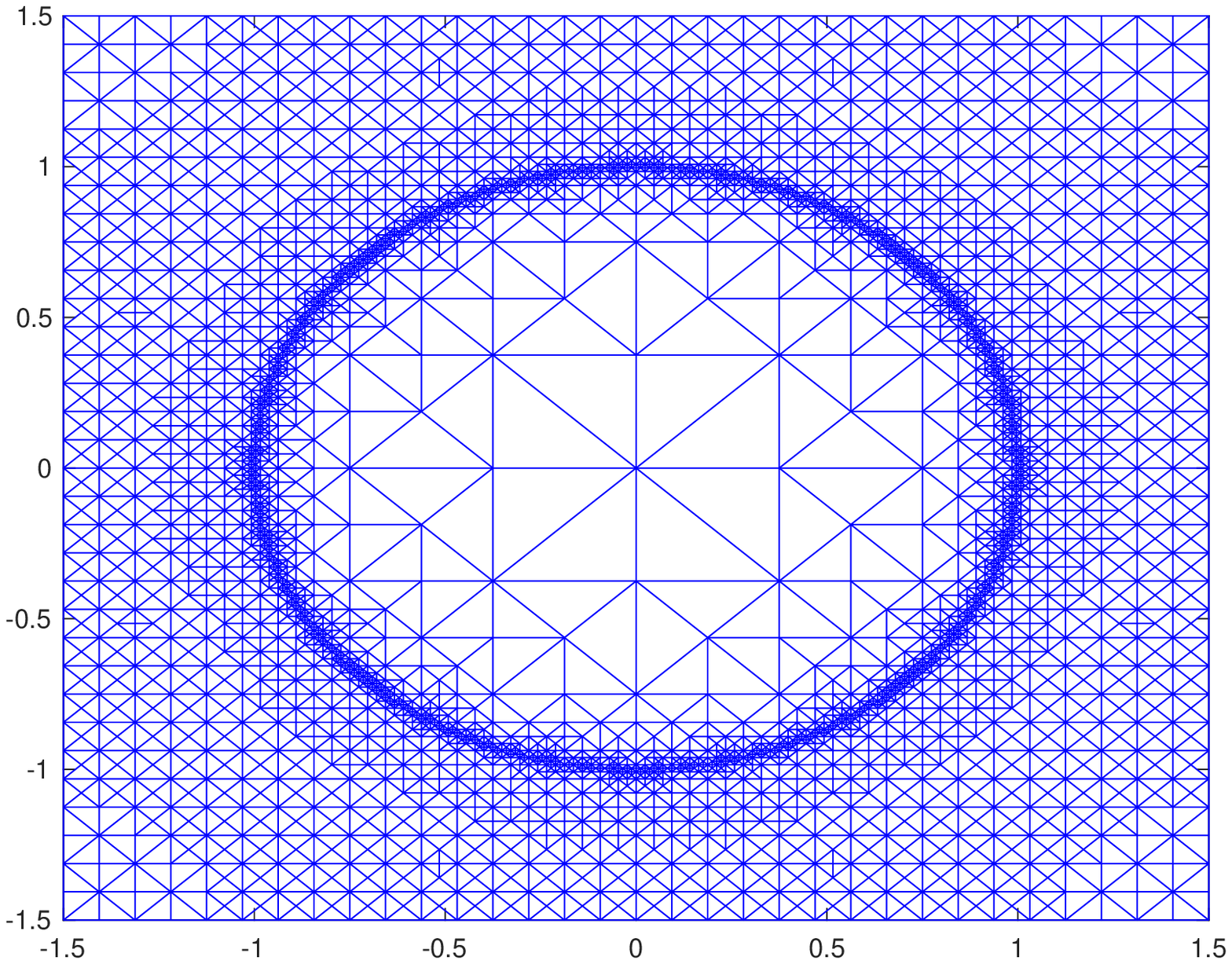}
		\caption{$\mathbb{P}_2$ (Integral Constraints)}
		\label{fig:fig4a}
	\end{subfigure}%
	\caption{Adaptive mesh (SIPG) for Example \ref{ex1}} \label{Fig2}
\end{figure}

\begin{example} \label{ex1}
	We consider a constant obstacle $\chi=0$ on the square $\Omega=(\frac{3}{2},-\frac{3}{2})^2$ (taken from the article \cite{bartels2004averaging})  and $f=-2$. The continuous solution of inequality \eqref{SSSS} is defined by
	\begin{equation*}
	u := \left\{ \begin{array}{ll} r^2/2-\text{ln}(r)-1/2, & r\geq 1\\\\
	0, & \text{otherwise},
	\end{array}\right.
	\end{equation*}
	where
	$r^2=x^2+y^2$ for $(x,y)\in \R^2$. Figure \ref{fig:fig1} shows the convergence behavior of the error estimator and the true error and  Figure \ref{fig:fig2} provides the efficiency indices of the error estimator for both SIPG and NIPG methods using $\mathbb{P}_1$-Integral constraints in the discrete convex set (recall equation \eqref{dset1}). We observe that the error and the estimator are converging at the optimal rate (1/DOFs) using $\mathbb{P}_1$ finite elements (with integral constraints) where DOFs stands for the degrees of freedom. For the $\mathbb{P}_2$ cases, Figures \ref{Fig3} and \ref{Fig4} describe that both error and estimator converge optimally with rate (DOFs)$^{-\frac{3}{2}}$ for SIPG and NIPG methods, respectively.  In these figures, we also depicted the efficiency indices for both DG methods (SIPG and NIPG) with quadratic elements. Lastly, adaptive mesh refinement at certain levels for SIPG method are shown in Figure \ref{Fig2} for $\mathbb{P}_2$ case. We observe that the mesh refinement is much higher near the free boundary. 

\end{example}
	\begin{figure}
	\begin{subfigure}[b]{0.4\textwidth}
		\includegraphics[width=\linewidth]{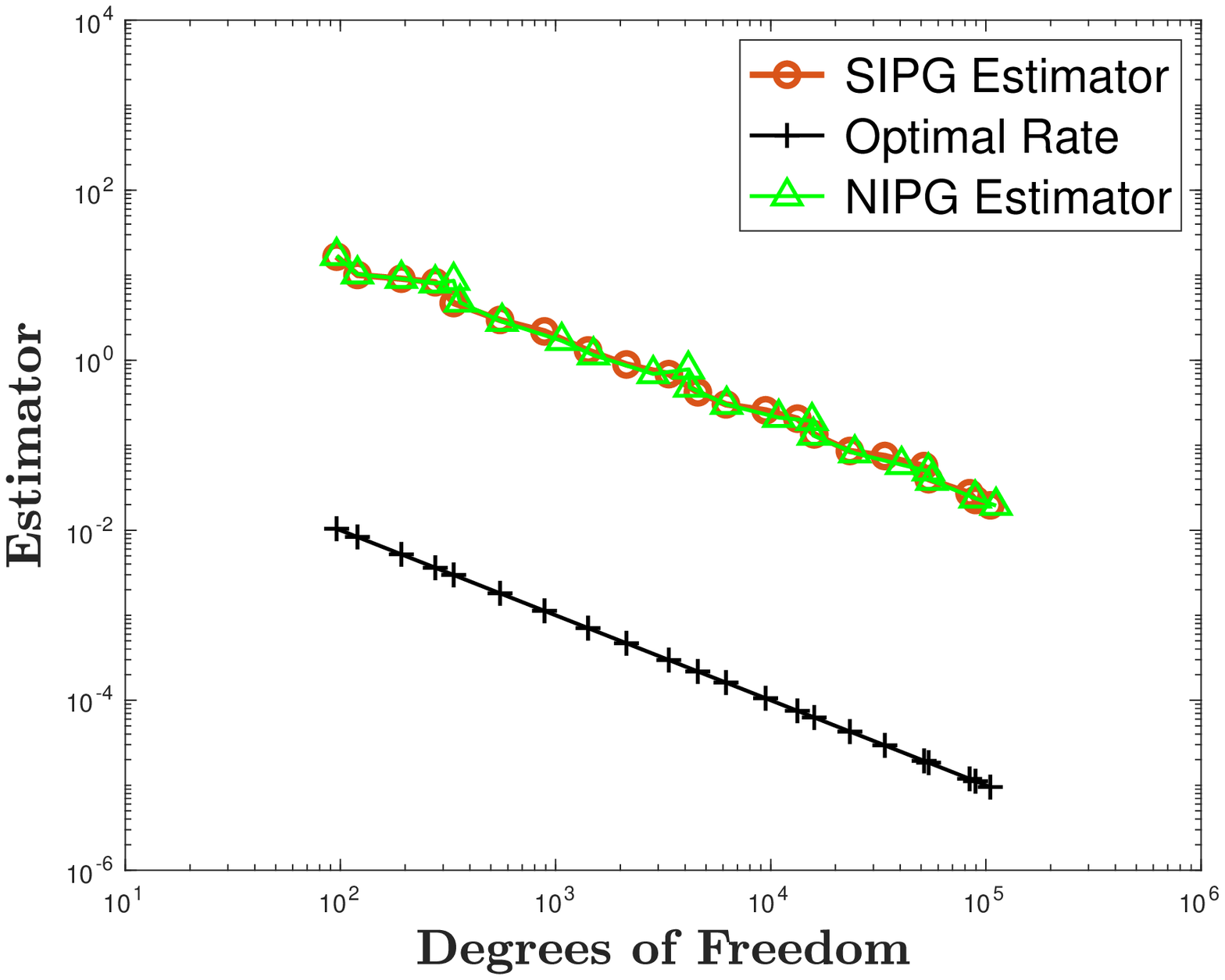}
		\label{fig:fig16}
	\end{subfigure}
	\begin{subfigure}[b]{0.4\textwidth}
	\includegraphics[width=\linewidth]{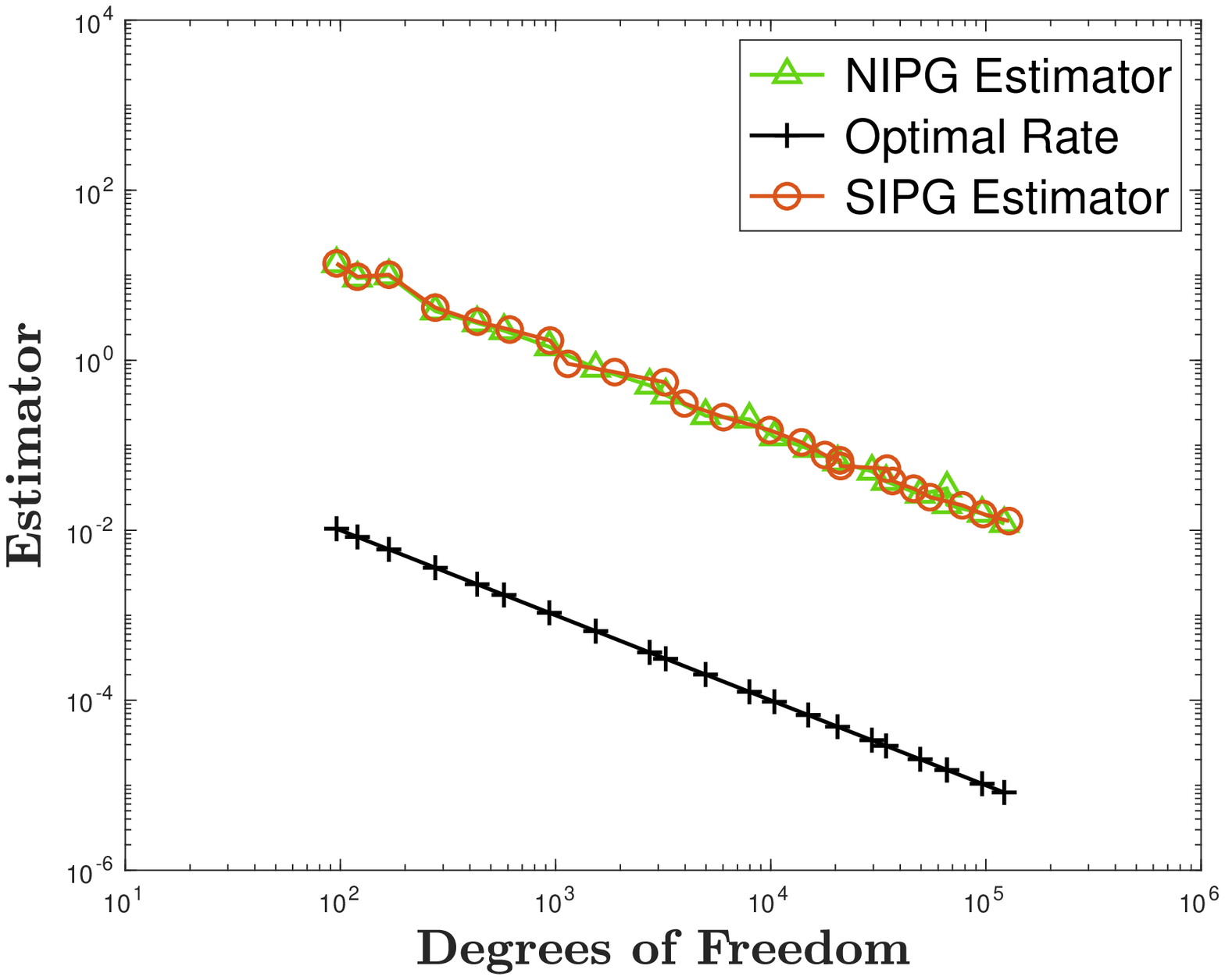}
	\label{fig:fig15}
\end{subfigure}
	\caption{Error and Estimator ($\mathbb{P}_1$- Integral Constraints) for Example \ref{ex3} for $f=-15$ and $f=0$. }\label{Fig8}
\end{figure} 	
	\begin{figure}
	\begin{subfigure}[b]{0.4\textwidth}
		\includegraphics[width=\linewidth]{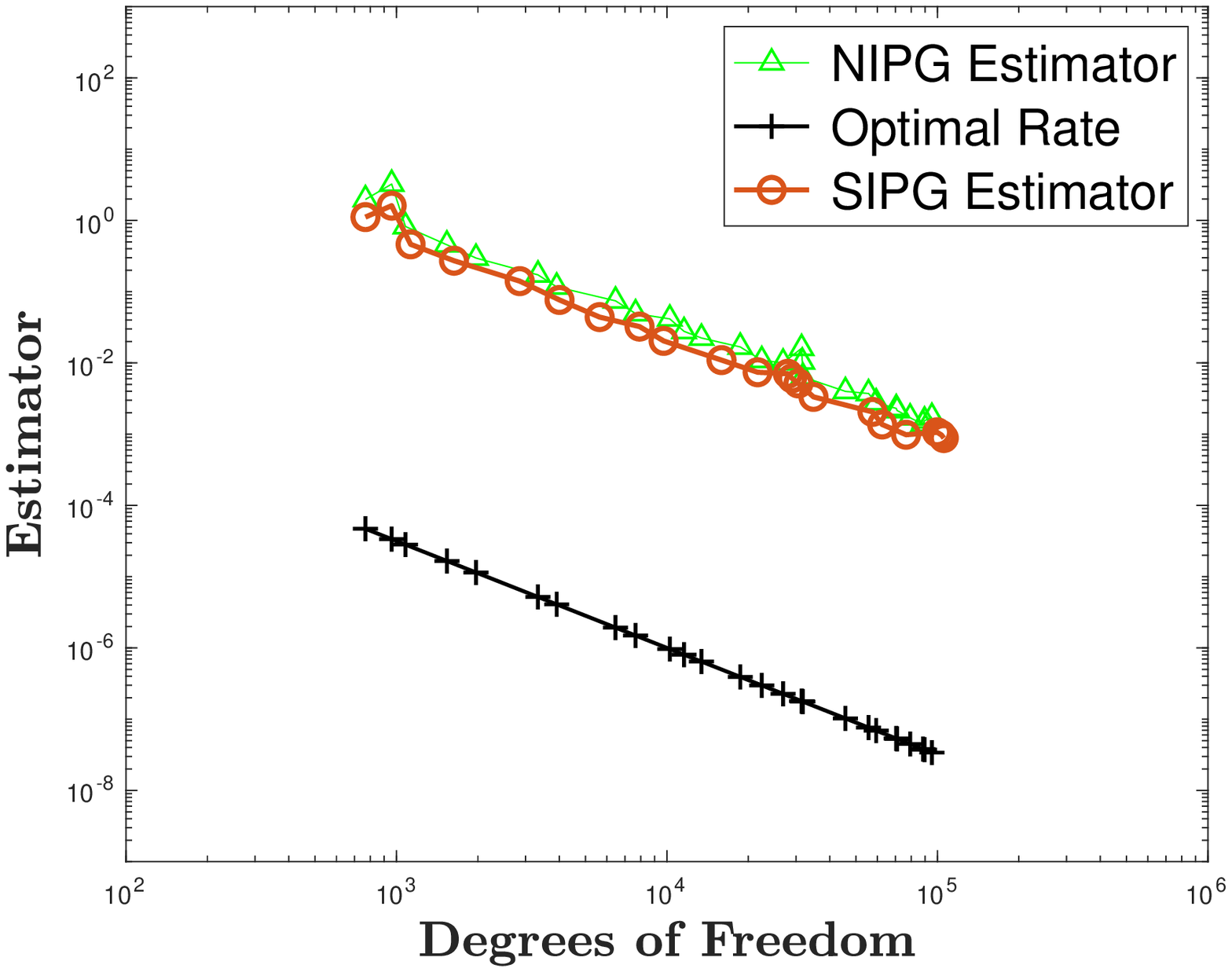}
		\label{fig:fig13}
	\end{subfigure}
	\begin{subfigure}[b]{0.4\textwidth}
		\includegraphics[width=\linewidth]{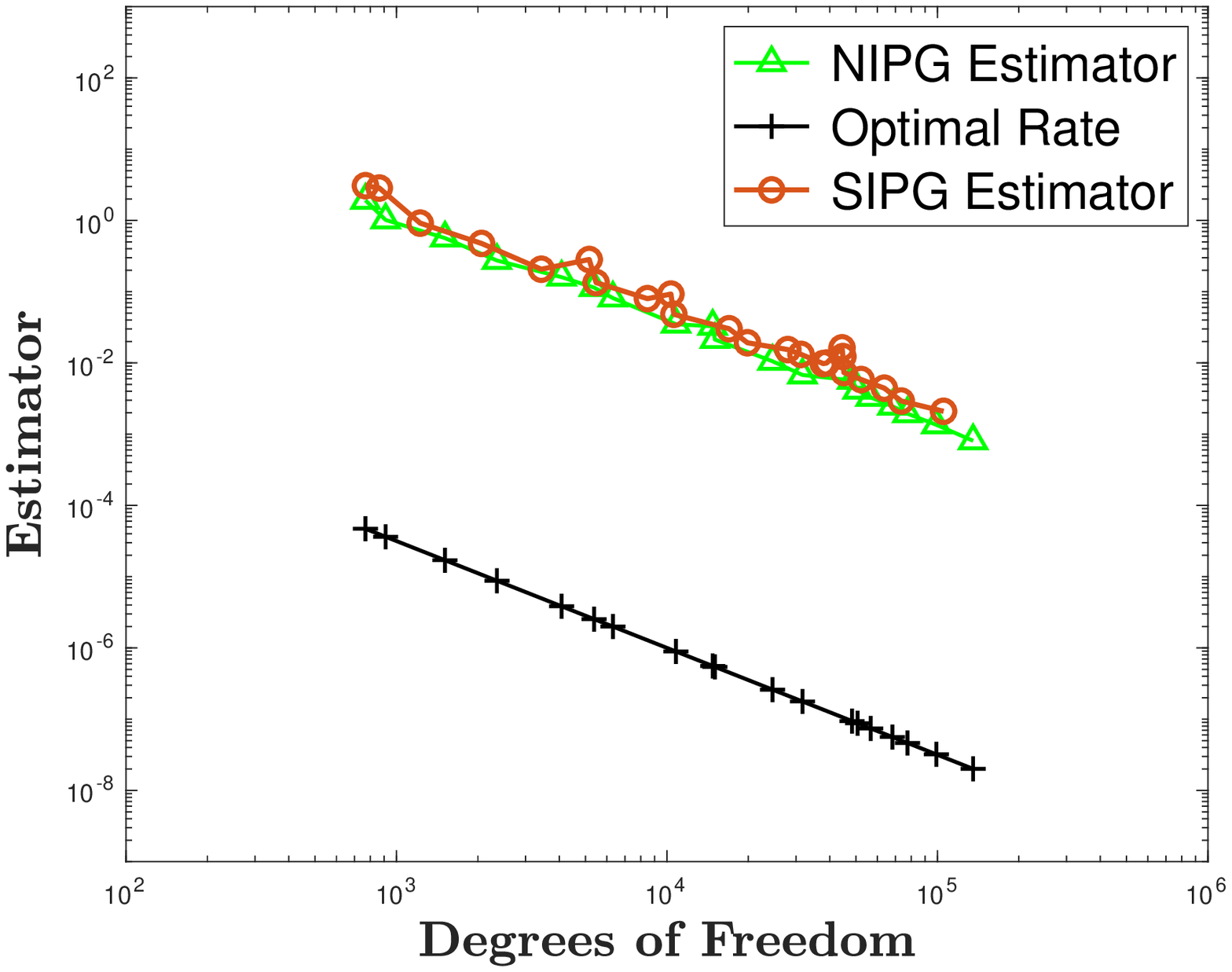}
		\label{fig:fig14}
	\end{subfigure}
	\caption{Error and Estimator ($\mathbb{P}_2$- Integral Constraints) for Example \ref{ex3} for $f=-15$ and $f=0$. }\label{Fig9}
\end{figure} 
\begin{example} \label{ex3} In this example, let $\Omega=(-2,2) \times (-1,1)$ and $\chi= 10 - 6 (x^2-1)^2-20(r^2-x^2)$ where $r^2=x^2+y^2~~ \forall (x,y) \in \Omega$. The exact solution is not known for this example \cite{nochetto2005fully}. Figure \ref{Fig8} illustrates the optimal convergence of the error estimator $\eta_h$ for the load functions $f=0$ and $f=-15$ using  integral constraints for linear $(\mathbb{P}_1)$ SIPG and NIPG methods.  In Figure \ref{Fig9} and \ref{Fig10}, we observe that the estimator $\eta_h$ converges with optimal rate (DOFs)$^{-\frac{3}{2}}$ for ($\mathbb{P}_2$) quadratic SIPG and NIPG methods, respectively. The adaptive meshes at refinement level 21 are shown in Figure \ref{Fig11} and \ref{Fig12} for quadratic SIPG method. We observe that the graph of the obstacle can be viewed as two hills connected by a saddle. As expected, there is a change in the contact region with the increase of load function $f$.
\end{example}
\begin{figure}
\begin{subfigure}[b]{0.4\textwidth}
	\includegraphics[width=\linewidth]{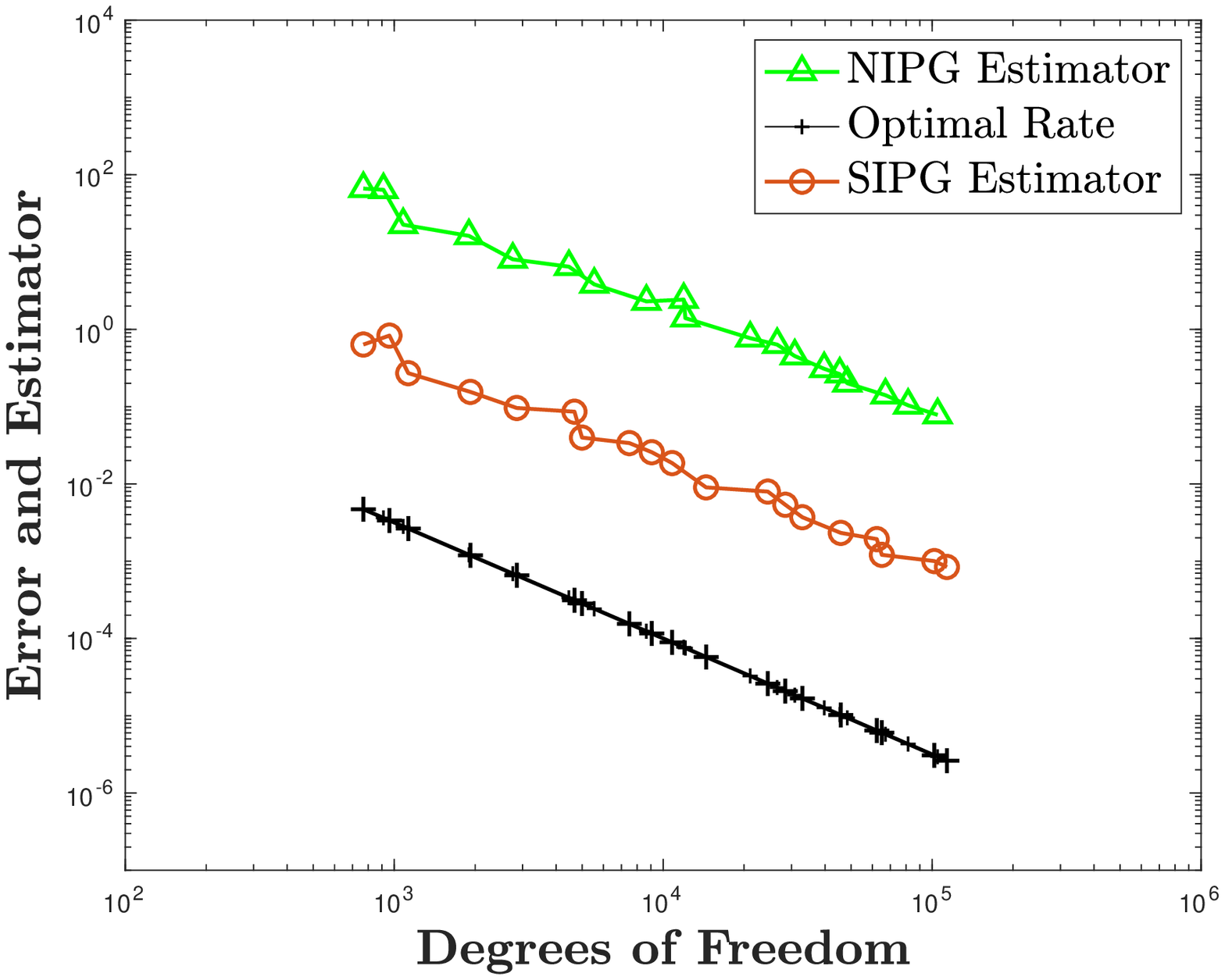}
	\label{fig:fig1188}
\end{subfigure}
\begin{subfigure}[b]{0.4\textwidth}
	\includegraphics[width=\linewidth]{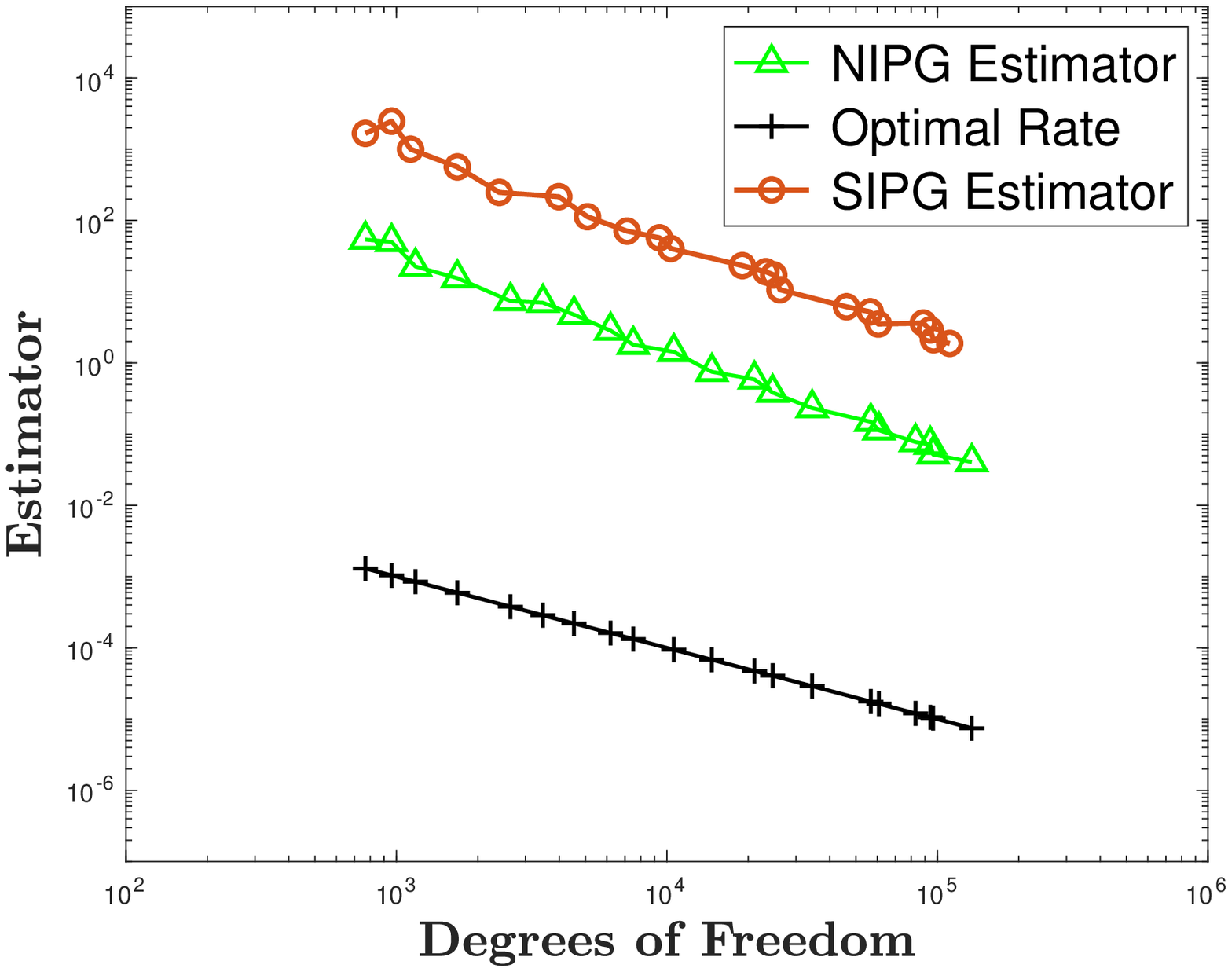}
	\label{fig:fig17}
\end{subfigure}
\caption{Error and Estimator ($\mathbb{P}_2$- Quadrature point Constraints) for Example \ref{ex3} for $f=-15$ and $f=0$ }\label{Fig10}
\end{figure}
\begin{figure}
	\begin{subfigure}[b]{0.4\textwidth}
		\includegraphics[width=\linewidth]{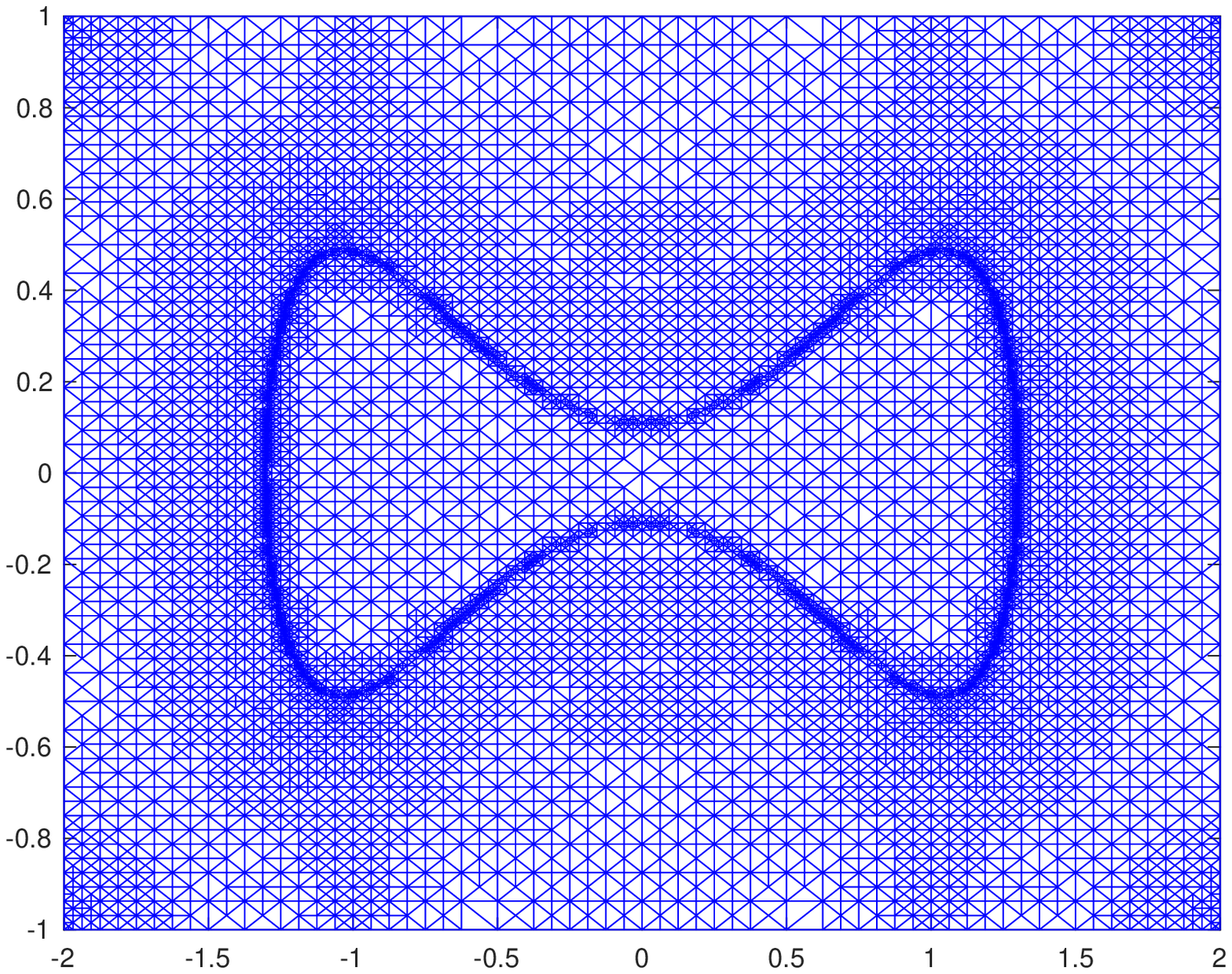}
		\caption{SIPG $(f=-15)$}
		\label{fig:fig1111}
	\end{subfigure}
	\begin{subfigure}[b]{0.4\textwidth}
		\includegraphics[width=\linewidth]{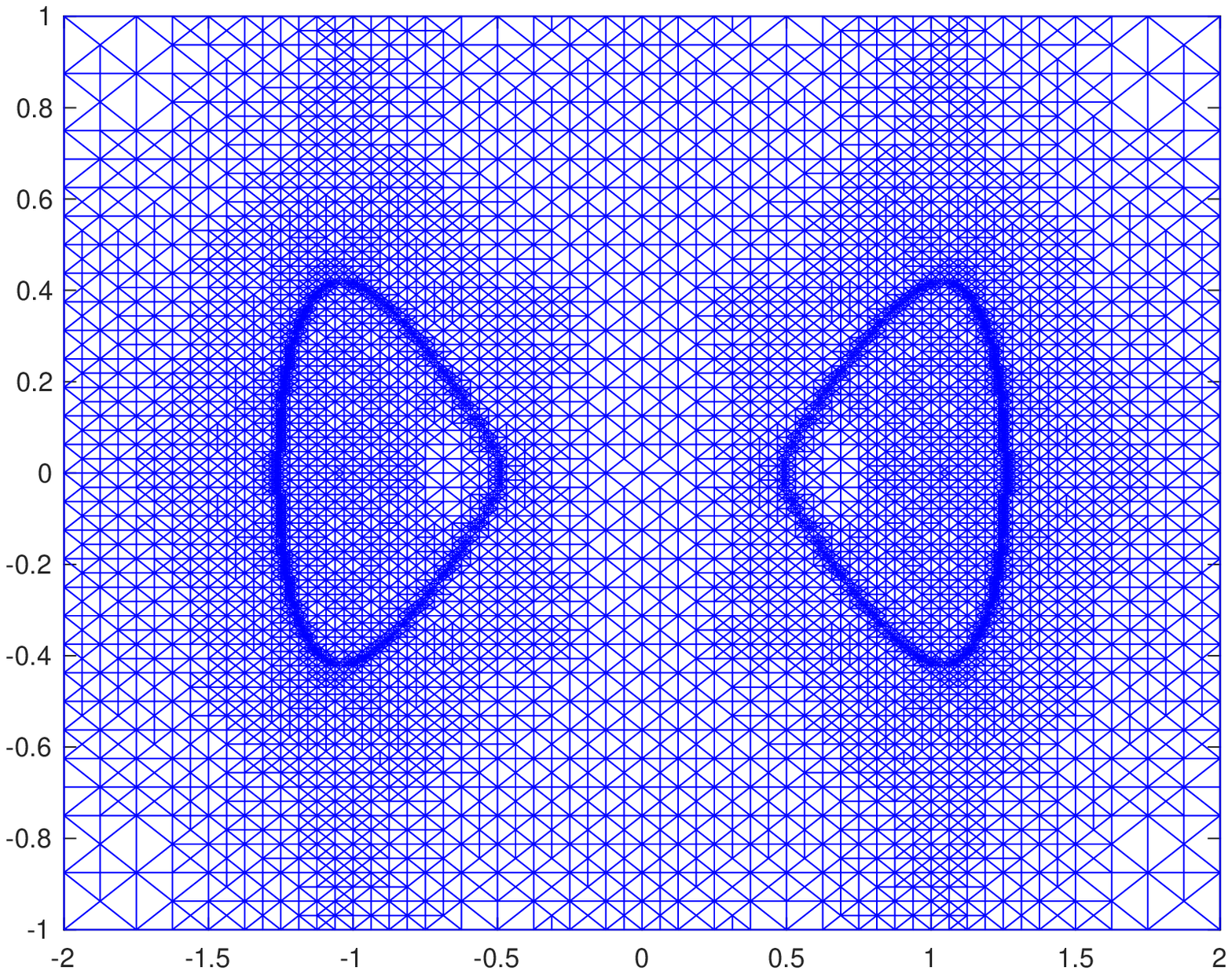}
		\caption{SIPG $(f=0)$}
		\label{fig:fig12128}
	\end{subfigure}
	\caption{Adaptive mesh ($\mathbb{P}_2$- Integral Constraints) for Example \ref{ex3}}\label{Fig11}
\end{figure} 
\begin{figure}
	\begin{subfigure}[b]{0.4\textwidth}
		\includegraphics[width=\linewidth]{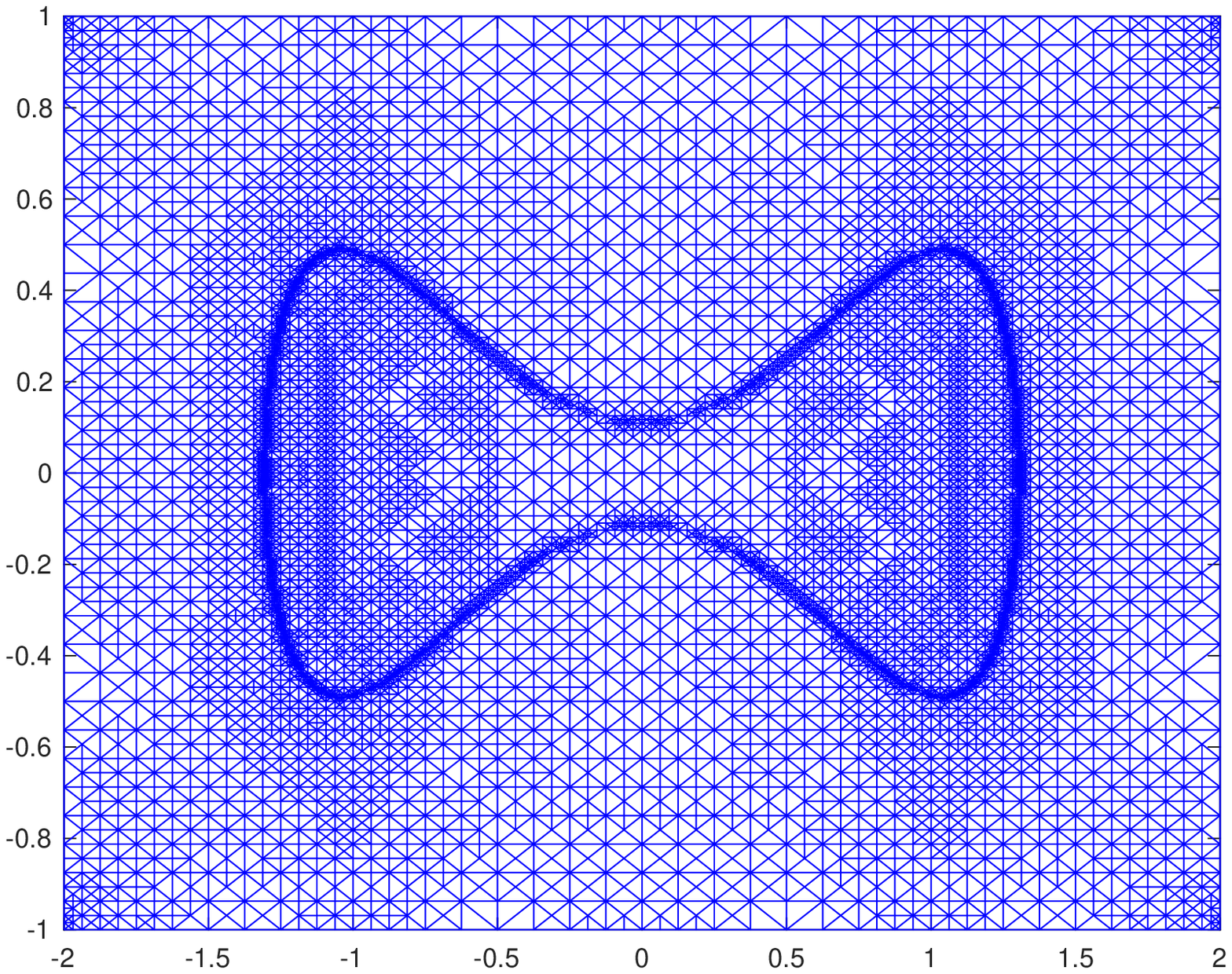}
			\caption{SIPG $(f=-15)$}
		\label{fig:fig117}
	\end{subfigure}
	\begin{subfigure}[b]{0.4\textwidth}
		\includegraphics[width=\linewidth]{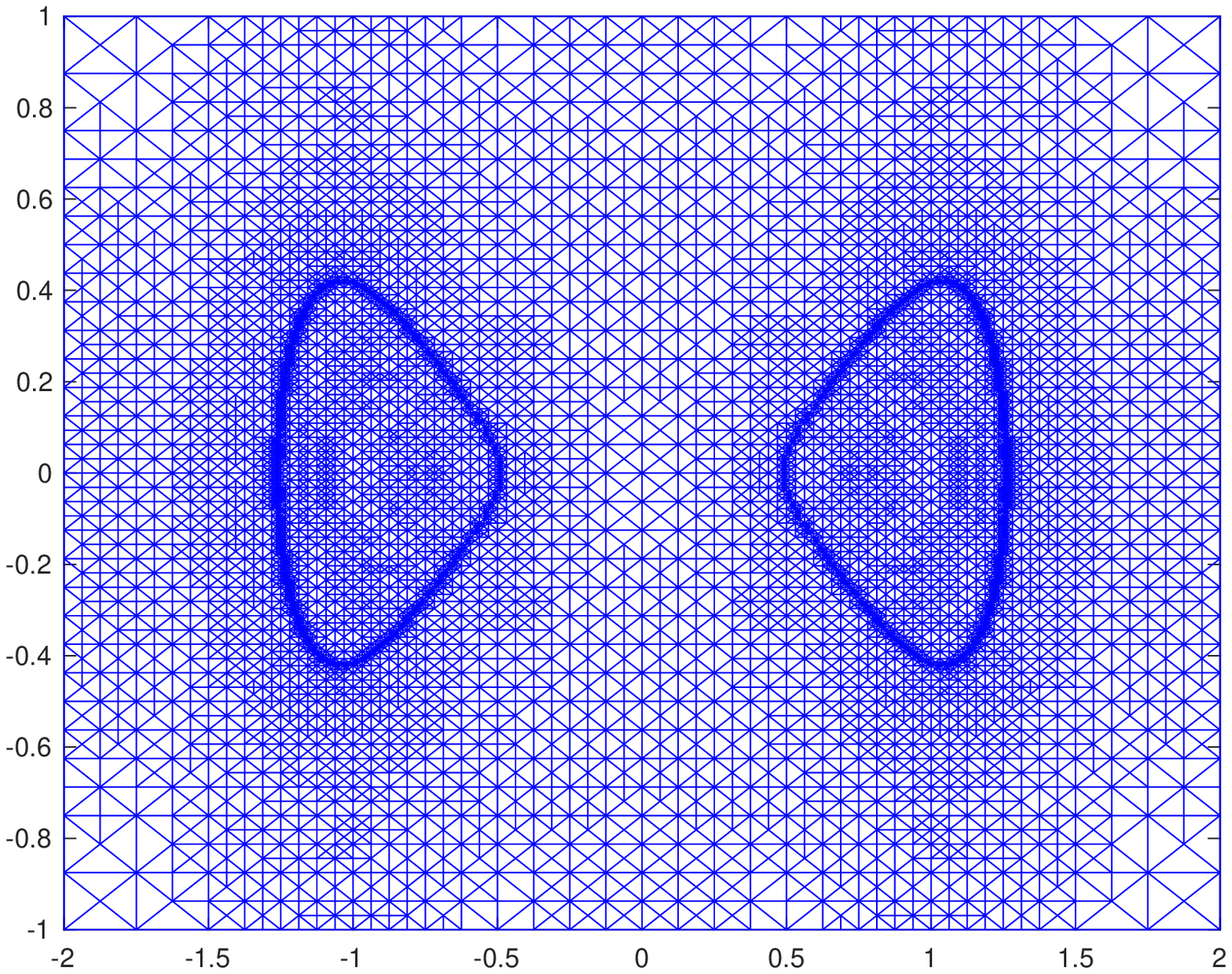}
			\caption{SIPG $(f=0)$}
		\label{fig:fig118}
	\end{subfigure}
	\caption{Adaptive mesh ($\mathbb{P}_2$- Quadrature point Constraints) for Example \ref{ex3}}\label{Fig12}
\end{figure} 


%
%

\bibliographystyle{unsrt}
\bibliography{rohiii}

\begin{thebibliography}{10}

\bibitem{gaddam2021two}
Sharat Gaddam, Thirupathi Gudi, and Kamana Porwal.
\newblock {Two new approaches for solving elliptic obstacle problems using
  discontinuous Galerkin methods}.
\newblock {\em BIT Numerical Mathematics}, 62:89--124, 2022.

\bibitem{rodrigues1987obstacle}
J-F Rodrigues.
\newblock {\em {Obstacle problems in mathematical physics}}.
\newblock Elsevier, 1987.

\bibitem{Falk:1974:VI}
R.~S. Falk.
\newblock Error estimates for the approximation of a class of variational
  inequalities.
\newblock {\em Math. Comp.}, 28:963--971, 1974.

\bibitem{brezzi1977error}
Franco Brezzi, William~W Hager, and Pierre-Arnaud Raviart.
\newblock {Error estimates for the finite element solution of variational
  inequalities}.
\newblock {\em Numerische Mathematik}, 28(4):431--443, 1977.

\bibitem{glowinski1984numerical}
Roland Glowinski.
\newblock {Numerical Methods for Nonlinear Variational Problems}.
\newblock {\em Numerical Methods for Nonlinear Variational Problems:}, 1984.

\bibitem{hoppe1994adaptive}
Ronald~HW Hoppe and Ralf Kornhuber.
\newblock {Adaptive multilevel methods for obstacle problems}.
\newblock {\em SIAM journal on numerical analysis}, 31(2):301--323, 1994.

\bibitem{brenner2012finite}
Susanne Brenner, Li-yeng Sung, and Yi~Zhang.
\newblock {Finite element methods for the displacement obstacle problem of
  clamped plates}.
\newblock {\em Mathematics of Computation}, 81(279):1247--1262, 2012.

\bibitem{verfurth1996review}
R{\"u}diger Verf{\"u}rth.
\newblock {A review of a posteriori error estimation}.
\newblock In {\em and Adaptive Mesh-Refinement Techniques, Wiley \& Teubner}.
  Citeseer, 1996.

\bibitem{ainsworth2011posteriori}
Mark Ainsworth and J~Tinsley Oden.
\newblock {\em {A posteriori error estimation in finite element analysis}},
  volume~37.
\newblock John Wiley \& Sons, 2011.

\bibitem{bangerth2013adaptive}
Wolfgang Bangerth and Rolf Rannacher.
\newblock {\em {Adaptive finite element methods for differential equations}}.
\newblock Birkh{\"a}user, 2013.

\bibitem{pollock2012convergence}
Sara Pollock.
\newblock {\em {Convergence of goal-oriented adaptive finite element methods}}.
\newblock PhD thesis, UC San Diego, 2012.

\bibitem{johnson1992adaptive}
Claes Johnson.
\newblock {Adaptive finite element methods for the obstacle problem}.
\newblock {\em Mathematical Models and Methods in Applied Sciences},
  2(04):483--487, 1992.

\bibitem{chen2000residual}
Zhiming Chen and Ricardo~H Nochetto.
\newblock {Residual type a posteriori error estimates for elliptic obstacle
  problems}.
\newblock {\em Numerische Mathematik}, 84(4):527--548, 2000.

\bibitem{veeser2001efficient}
Andreas Veeser.
\newblock {Efficient and reliable a posteriori error estimators for elliptic
  obstacle problems}.
\newblock {\em SIAM journal on numerical analysis}, 39(1):146--167, 2001.

\bibitem{bartels2004averaging}
S{\"o}ren Bartels and Carsten Carstensen.
\newblock {Averaging techniques yield reliable a posteriori finite element
  error control for obstacle problems}.
\newblock {\em Numerische Mathematik}, 99(2):225--249, 2004.

\bibitem{braess2005posteriori}
Dietrich Braess.
\newblock A posteriori error estimators for obstacle problems--another look.
\newblock {\em Numerische Mathematik}, 101(3):415--421, 2005.

\bibitem{NPZ:2010:VI}
T.~V.~Petersdorff R.~Nochetto and C.~S. {Zhang}.
\newblock A posteriori error analysis for a class of integral equations and
  variational inequalities.
\newblock {\em Numer. Math.}, 116:519--552, 2010.

\bibitem{GG:2018:InObst}
S.~Gaddam and T.~{Gudi}.
\newblock Inhomogeneous dirichlet boundary condition in the a posteriori error
  control of the obstacle problem.
\newblock {\em Comput. Math. Appl.}, 75:2311--2327, 2018.

\bibitem{wang2002quadratic}
Lie-heng Wang.
\newblock {On the quadratic finite element approximation to the obstacle
  problem}.
\newblock {\em Numerische Mathematik}, 92(4):771--778, 2002.

\bibitem{gudi2015reliable}
Thirupathi Gudi and Kamana Porwal.
\newblock {A reliable residual based a posteriori error estimator for a
  quadratic finite element method for the elliptic obstacle problem}.
\newblock {\em Computational Methods in Applied Mathematics}, 15(2):145--160,
  2015.

\bibitem{gaddam2018bubbles}
Sharat Gaddam and Thirupathi Gudi.
\newblock {Bubbles enriched quadratic finite element method for the 3d-elliptic
  obstacle problem}.
\newblock {\em Computational Methods in Applied Mathematics}, 18(2):223--236,
  2018.

\bibitem{dari1999maximum}
Enzo Dari, Ricardo~G Dur{\'a}n, and Claudio Padra.
\newblock {Maximum norm error estimators for three-dimensional elliptic
  problems}.
\newblock {\em SIAM Journal on Numerical Analysis}, 37(2):683--700, 1999.

\bibitem{demlow2007local}
Alan Demlow.
\newblock {Local a posteriori estimates for pointwise gradient errors in finite
  element methods for elliptic problems}.
\newblock {\em Mathematics of computation}, 76(257):19--42, 2007.

\bibitem{demlow2012pointwise}
Alan Demlow and Emmanuil~H Georgoulis.
\newblock {Pointwise a posteriori error control for discontinuous Galerkin
  methods for elliptic problems}.
\newblock {\em SIAM Journal on Numerical Analysis}, 50(5):2159--2181, 2012.

\bibitem{nochetto2003pointwise}
Ricardo~H Nochetto, Kunibert~G Siebert, and Andreas Veeser.
\newblock {Pointwise a posteriori error control for elliptic obstacle
  problems}.
\newblock {\em Numerische Mathematik}, 95(1):163--195, 2003.

\bibitem{nochetto2005fully}
Ricardo~H Nochetto, Kunibert~G Siebert, and Andreas Veeser.
\newblock {Fully localized a posteriori error estimators and barrier sets for
  contact problems}.
\newblock {\em SIAM journal on numerical analysis}, 42(5):2118--2135, 2005.

\bibitem{KP:2022:Obstacle}
Rohit Khandelwal and Kamana Porwal.
\newblock Pointwise a posteriori error analysis of quadratic finite element
  method for the elliptic obstacle problem.
\newblock {\em Communicated}.

\bibitem{gudi2014posteriori}
Thirupathi Gudi and Kamana Porwal.
\newblock {A posteriori error control of discontinuous Galerkin methods for
  elliptic obstacle problems}.
\newblock {\em Mathematics of Computation}, 83(286):579--602, 2014.

\bibitem{TG:2014:VIDG1}
T.~Gudi and K.~{Porwal}.
\newblock A remark on the a posteriori error analysis of discontinuous galerkin
  methods for obstacle problem.
\newblock {\em Comput. Meth. Appl. Math.}, 14:71--87, 2014.

\bibitem{BGP:2021}
Blanca Ayuso~de Dios, Thirupathi Gudi, and Kamana Porwal.
\newblock {Pointwise a posteriori error analysis of a discontinuous Galerkin
  method for the elliptic obstacle problem}.
\newblock {\em Communicated}.

\bibitem{brenner1996two}
Susanne Brenner.
\newblock Two-level additive schwarz preconditioners for nonconforming finite
  element methods.
\newblock {\em Mathematics of Computation}, 65(215):897--921, 1996.

\bibitem{s1989topics}
S.~Kesavan.
\newblock {\em Topics in functional analysis and applications}.
\newblock John Wiley \& Sons, 1989.

\bibitem{glowinski1980numerical}
Roland Glowinski.
\newblock {\em {Numerical methods for nonlinear variational problems}}.
\newblock Tata Institute of Fundamental Research, 1980.

\bibitem{kinderlehrer2000introduction}
David Kinderlehrer and Guido Stampacchia.
\newblock {\em {An introduction to variational inequalities and their
  applications}}.
\newblock SIAM, 2000.

\bibitem{frehse1978smoothness}
Jens Frehse.
\newblock On the smoothness of solutions of variational inequalities with
  obstacles.
\newblock {\em Proc. Banach Center Semester on Partial Differential Equations},
  10:81--128, 1978.

\bibitem{caffareli1980potential}
LA~Caffareli and David Kinderlehrer.
\newblock Potential methods in variational inequalities.
\newblock {\em Journal d’Analyse Math{\'e}matique}, 37(1):285--295, 1980.

\bibitem{nochetto1995pointwise}
Ricardo~H Nochetto.
\newblock {Pointwise a posteriori error estimates for elliptic problems on
  highly graded meshes}.
\newblock {\em Mathematics of computation}, 64(209):1--22, 1995.

\bibitem{demlow2016maximum}
Alan Demlow and Natalia Kopteva.
\newblock {Maximum-norm a posteriori error estimates for singularly perturbed
  elliptic reaction-diffusion problems}.
\newblock {\em Numerische Mathematik}, 133(4):707--742, 2016.

\bibitem{gruter1982green}
Michael Gr{\"u}ter and Kjell-Ove Widman.
\newblock {The Green function for uniformly elliptic equations}.
\newblock {\em Manuscripta Mathematica}, 37(3):303--342, 1982.

\bibitem{hofmann2007green}
Steve Hofmann and Seick Kim.
\newblock The green function estimates for strongly elliptic systems of second
  order.
\newblock {\em manuscripta mathematica}, 124(2):139--172, 2007.

\bibitem{ciarlet2002finite}
Philippe~G Ciarlet.
\newblock {\em {The finite element method for elliptic problems}}.
\newblock SIAM, 2002.

\bibitem{brenner2007mathematical}
{\em {The mathematical theory of finite element methods}, author={Brenner,
  Susanne and Scott, Ridgway}}, volume~15.
\newblock Springer Science \& Business Media, 2007.

\bibitem{evans1998partial}
Lawrence~C Evans.
\newblock Partial differential equations.
\newblock {\em Graduate studies in mathematics}, 19(4):--, 1998.

\bibitem{scott1990finite}
L~Ridgway Scott and Shangyou Zhang.
\newblock {Finite element interpolation of nonsmooth functions satisfying
  boundary conditions}.
\newblock {\em Mathematics of Computation}, 54(190):483--493, 1990.

\bibitem{brenner1999convergence}
Susanne Brenner.
\newblock {Convergence of nonconforming multigrid methods without full elliptic
  regularity}.
\newblock {\em Mathematics of computation}, 68(225):25--53, 1999.

\bibitem{wang2011discontinuous}
Fei Wang, Weimin Han, and Xiaoliang Cheng.
\newblock Discontinuous galerkin methods for solving the signorini problem.
\newblock {\em IMA journal of numerical analysis}, 31(4):1754--1772, 2011.

\bibitem{nochetto2006pointwise}
Ricardo~H Nochetto, Alfred Schmidt, Kunibert~G Siebert, and Andreas Veeser.
\newblock Pointwise a posteriori error estimates for monotone semi-linear
  equations.
\newblock {\em Numerische Mathematik}, 104(4):515--538, 2006.

\bibitem{hintermuller2002primal}
Michael Hinterm{\"u}ller, Kazufumi Ito, and Karl Kunisch.
\newblock The primal-dual active set strategy as a semismooth newton method.
\newblock {\em SIAM Journal on Optimization}, 13(3):865--888, 2002.

\end{thebibliography}
\end{document}